\documentclass[a4paper, 11pt]{amsart}

\usepackage{enumitem, amsmath, amssymb, amsthm, color, tikz, caption, tikz-cd, hyperref, mathrsfs, calc, mathtools, bigints, verbatim, stmaryrd, bm, scalerel, mathabx}
\usepackage[capitalize]{cleveref}

\newcommand{\lb}{\left(}
\newcommand{\rb}{\right)}
\newcommand{\eps}{\varepsilon}
\newcommand{\Q}{\mathbb{Q}}
\newcommand{\R}{\mathbb{R}}
\newcommand{\C}{\mathbb{C}}
\newcommand{\Z}{\mathbb{Z}}
\newcommand{\K}{\mathbb{K}}
\renewcommand{\P}{\mathbb{P}}
\newcommand{\pd}{\partial}
\renewcommand{\phi}{\varphi}
\newcommand{\ip}[2]{\langle #1, #2 \rangle}
\let\oldref\ref
\renewcommand{\eqref}[1]{\textup{(\oldref{#1})}}
\newcommand{\SU}{\mathrm{SU}}
\newcommand{\U}{\mathrm{U}}
\newcommand{\PSU}{\mathrm{PSU}}
\newcommand{\CO}{\operatorname{\mathcal{CO}^0}}

\newcommand{\Sym}{\operatorname{Sym}}
\newcommand{\qcup}{\mathbin{\Asterisk}}
\newcommand{\Char}{\operatorname{char}}
\newcommand{\begen}{\begin{enumerate}[label=(\roman*), align=right]}

\newcommand{\LCEL}{L_{\text{AOC}}} 
\newcommand{\TCh}{T_\mathrm{Ch}}
\newcommand{\Gr}{\mathrm{Gr}}
\renewcommand{\sslash}{/\mkern-6mu/}
\DeclareMathOperator*{\bigqcup}{\scalerel*{\qcup}{\sum}}
\newcommand{\pr}{\operatorname{pr}}
\newcommand{\rank}{\operatorname{rank}}
\newcommand{\diff}{\mathrm{d}}
\newcommand{\Hom}{\operatorname{Hom}}
\newcommand{\Tr}{\operatorname{Tr}}
\newcommand{\lspan}[1]{\left\langle #1 \right\rangle}
\newcommand{\id}{\operatorname{id}}

\let\originalleft\left
\let\originalright\right
\renewcommand{\left}{\mathopen{}\mathclose\bgroup\originalleft}
\renewcommand{\right}{\aftergroup\egroup\originalright}

\theoremstyle{plain}
\newtheorem{thm}{Theorem}[section]
\newtheorem{lem}[thm]{Lemma}
\newtheorem{prop}[thm]{Proposition}

\newtheorem{mthm}{Theorem}

\theoremstyle{remark}
\newtheorem{rmkx}[thm]{Remark}
\newenvironment{rmk}
  {\pushQED{\qed}\rmkx}
  {\popQED\endrmkx}

\theoremstyle{plain}

\theoremstyle{definition}
\newtheorem{defnx}[thm]{Definition}
\newenvironment{defn}
  {\pushQED{\qed}\defnx}
  {\popQED\enddefnx}

\title{A monotone Lagrangian casebook}
\author{Jack Smith}
\address{Department of Mathematics\\ University College London\\ Gower Street\\ London\\ WC1E  6BT}
\email{jack.smith@ucl.ac.uk}

\begin{document}

\begin{abstract}
We present an array of new calculations in Lagrangian Floer theory which demonstrate observations relating to symplectic reduction, grading periodicity, and the closed--open map.  We also illustrate Perutz's symplectic Gysin sequence and the quilt theory of Wehrheim and Woodward.
\end{abstract}

\maketitle


\section{Introduction}

Given a monotone Lagrangian submanifold $L$ of a symplectic manifold $X$, a fundamental invariant is its Floer cohomology.  This describes the endomorphisms of $L$ in the montone Fukaya category of $X$, and its non-vanishing indicates that $L$ cannot be displaced from itself by a Hamiltonian flow, but in general it is difficult to calculate.  In this note we make some simple observations about its properties and exploit them to study specific examples.  In addition, we give some explicit computations in $(\C\P^1)^3$ and $\C\P^2\times \C\P^1$ which demonstrate Perutz's symplectic Gysin sequence \cite{PerGysin} and Wehrheim--Woodward's quilt theory \cite{QuiltedHF}, for which there are few concrete calculations in the literature.  Throughout we pay particular attention to relative spin structures, and chase through the signs on which the answers are subtly dependent.


\subsection{Setup}
\label{sscSetup}

We assume that $X$ is compact or tame (convex or geometrically bounded) at infinity, and that $L$ is closed, connected and monotone, meaning that the Maslov index and area morphisms
\[
\mu : \pi_2(X, L) \rightarrow \Z \quad \text{and} \quad \omega : \pi_2(X, L) \rightarrow \R
\]
are positively proportional.  We also require that the minimal Maslov number $N_L \in \Z_{>0} \cup \{\infty\}$, meaning the positive generator of $\mu(\pi_2(X, L))$, is at least $2$.  If the characteristic of the coefficient field $\K$ is not $2$ then we assume that $L$ is orientable (which automatically makes $N_L$ even and hence at least $2$) and relatively spin, and endowed with a choice of relative spin structure $s$ as in \cref{sscRelSpin}.  Associated to $s$ is a background class $b$ in $H^2(X; \Z/2)$.  The choice of $s$ orients the moduli spaces of pseudoholomorphic discs with boundary on $L$, whose counts are the key ingredients in Floer theory.  $L$ may also be equipped with a flat line bundle $\mathcal{L}$ over $\K$.

We denote the triple $(L, s, \mathcal{L})$ by $L^\flat$ and call it a \emph{monotone Lagrangian brane}.  It has a Floer cohomology algebra $HF^*(L^\flat, L^\flat; \Lambda)$ over the Laurent polynomial ring $\Lambda = \K[T^{\pm 1}]$, in which the `Novikov variable' $T$ has degree $N_L$.  The construction and basic properties of this algebra are described in \cite{BCQS}, where it is called the Lagrangian quantum homology of $L$ (in contrast to \cite{BCQS} we use cohomological grading).  This is our main object of study.

We summarise some of its properties in \cref{secPrerequisites}, but for now it suffices to point out the following.  There is a unital $\Lambda$-algebra homomorphism, the \emph{(length-zero) closed--open string map}
\[
\CO :QH^*(X, b; \Lambda) \rightarrow HF^*(L^\flat, L^\flat; \Lambda),
\]
from the quantum cohomology of $X$ (with background class $b$), and this induces the `quantum module action' of $QH^*$ on $HF^*$.  There is also a multiplicative spectral sequence
\[
E_1 = H^*(L; \K) \otimes_\K \Lambda \implies HF^*(L^\flat, L^\flat; \Lambda),
\]
originally due to Oh \cite{OhSS}.  We say $L^\flat$ is \emph{wide} if we have $HF^*(L^\flat, L^\flat; \Lambda) \cong H^*(L; \K) \otimes_\K \Lambda$ as graded $\Lambda$-modules, and \emph{narrow} if $HF^*(L^\flat, L^\flat; \Lambda)=0$.


\subsection{Main results}

The bulk of the paper comprises worked examples, in which we compute $HF^*(L^\flat, L^\flat; \Lambda)$ as a $\Lambda$-algebra or -module for various Lagrangians $L$.  After the brief Floer theory review in \cref{secPrerequisites}, each section focuses on a different technique, and can be read independently of the others.  The only exception is that the examples studied in \cref{secQuilts} are defined in \cref{secGysin}.

In \cref{secReduction}, for each sequence $k_1, \dots, k_r$ of positive integers we construct a Lagrangian embedding of the flag variety $F(d_1, \dots, d_r)$ in the monotone product of Grassmannians $\Gr(k_1, n) \times \dots \times \Gr(k_r, n)$, where $d_j \coloneqq k_1+\dots+k_j$ and $n \coloneqq d_r$.  We equip it with a specific choice of relative spin structure and the trivial flat line bundle, and combining $\CO$ with knowledge of the quantum cohomology of Grassmannians we compute

\begin{mthm}[\cref{HFFlagVariety}]
\label{mHFFlagVariety}
The Floer cohomology algebra of this Lagrangian is
\[
\Lambda[c_{1,1}, \dots, c_{1,k_1}, c_{2,1}, \dots, c_{r, k_r}]\bigg/\bigqcup_{j=1}^r (1+c_{j,1}+\dots+c_{j,k_j}) = 1+T,
\]
where $\qcup$ denotes the Floer product, $c_{j,k}$ has degree $2k$, and the Novikov variable $T$ has degree $2n$.
\end{mthm}

The importance of the relatively spin condition becomes manifest in this calculation: if one tries to use a background class $b$ which does not admit a relative spin structure then one obtains relations which are only consistent in characteristic $2$.

This Lagrangian is constructed by symplectic reduction using a Hamiltonian action of $\U(k_1)\times\dots\times\U(k_r)$ on $\C^{n^2}$, and along the way we prove the following monotonicity transfer property for general symplectic reductions.

\begin{mthm}[\cref{lemRedMonotone}]
\label{mlemRedMonotone}
Suppose $X$ is a symplectic manifold carrying a Hamiltonian action of a compact connected Lie group $K$, such that $K$ acts freely on the zero set $\mu^{-1}(0)$ of the moment map.  If $L \subset X$ is a $K$-invariant monotone Lagrangian submanifold contained in $\mu^{-1}(0)$ then $L/K$ is monotone in $X \sslash K$.  In particular, $X\sslash K$ is (spherically) monotone.
\end{mthm}

This allows one to deduce monotonicity of the reduced manifold $X\sslash K$ or Lagrangian $L/K$ from monotonicity of $L$, which may be much simpler topologically.

In \cref{secFloerPoincare} we study periodicity in the grading of $HF^*$ using the discrete Fourier transform.  This gives constraints on the minimal Maslov number $N_L$ for certain Lagrangians:

\begin{mthm}[\cref{thmExteriorAlgebraNL}]
If $H^*(L; \K)$ is an exterior algebra on generators of degree $2k_1-1 \leq \dots \leq 2k_r-r$, and the quantum cohomology of $X$ contains an invertible element of even degree $q\leq 2k_r$ then $N_L$ is at most $2k_r$.  In particular, if $L$ is a monotone torus and $QH^*(X)$ contains an invertible element of degree $2$ then $N_L=2$.
\end{mthm}

By a more careful analysis we prove the following dichotomy for Lagrangian embeddings of $\PSU(n)$ of high minimal Maslov index.

\begin{mthm}[\cref{corLagPSU}]  Suppose $L$ is a Lagrangian embedding of $\PSU(n)$ in a closed monotone symplectic manifold $X$, whose quantum cohomology contains an invertible element of degree $2$.  Assume $n \geq 2$ and $N_L \geq 2n$.  Then we have $N_L=2n$ and either:
\begen
\item $\K$ has prime characteristic $p$ and $n$ is a power of $p$, in which case $L^\flat$ is wide for all line bundles $\mathcal{L}$ and all relative spin structures.
\item Otherwise $L^\flat$ is narrow over $\K$ for all such $\mathcal{L}$ and $s$.
\end{enumerate}
\end{mthm}

This completes partial computations \cite[Theorem 8, Proposition 27]{Iriyeh}, \cite[Example 7.2.3]{EL2} of a family of Lagrangian $\PSU(n)$'s in $\C\P^{n^2-1}$.

\Cref{secTrivialVBs} begins with the following easy result.

\begin{mthm}[\cref{COChern}]
\label{mCOChern}
If $E \rightarrow X$ is a complex vector bundle whose restriction to $L$ is trivial, and $j < N_L / 2$, then we have $\CO(c_j(E)) = 0$, where $c_j(E)$ is the $j$th Chern class of $E$.
\end{mthm}

Although this result is not especially noteworthy in itself, it is surprisingly useful, and we demonstrate this by considering a family of Lagrangians $L \subset \Gr(n-k, n) \times \C\P^{kn-1}$ diffeomorphic to the projective Stiefel manifold parametrising projective $k$-frames in $\C^n$.  These appear in \cite{ZapolskyGrassmannians} and specialise to the above family of $\PSU(n)$'s in the case $k=n$.  We show

\begin{mthm}[\cref{thmProjStiefel}]
\label{mthmProjStiefel}
If $p$ denotes $\Char \K$ and $p^r$ its greatest power dividing $n$ (interpreted as $1$ if $p=0$) then either: $k \leq p^r$, in which case $L$ is wide for all choices of relative spin structure and flat line bundle; or $k > p^r$, in which case $L$ is narrow for all such choices.
\end{mthm}

The wideness result was known by Zapolsky but the narrowness part is new.

Finally, in \cref{secGysin,secQuilts} we study a monotone Lagrangian embedding of $\mathrm{SO(3)}$ in $(\C\P^1)^3$ and of the lens space $L(4,1)$ in $\C\P^2\times\C\P^1$.  The former was studied in \cite[Section 5]{SmDCS}, where we showed that it is narrow unless $\Char \K$ is $3$ or $5$, depending on the choice of relative spin structure.  We also showed that it is wide in the $\Char \K = 3$ case.  We now show

\begin{mthm}[\cref{P13Wide}, \cref{sscQuiltI}]
The Lagrangian $\mathrm{SO}(3)$ in $(\C\P^1)^3$ is also wide in the $\Char \K = 5$ case.
\end{mthm}

We give two proofs of this result.  In \cref{sscGysinI} we use Perutz's Gysin sequence, viewing the Lagrangian as a circle bundle over the final $\C\P^1$ factor.  This expresses $HF^*$ as the cone on a quantum version of `cupping with the Euler class' $H^*(\C\P^1)\rightarrow H^*(\C\P^1)$.  In \cref{sscQuiltI} meanwhile, we use quilt theory (summarised in \cref{sscQuiltSummary}) to relate $L$ to the Chekanov torus in $\C\P^1\times \C\P^1$.  The Floer theory of the latter is well-understood, and using a result of Wehrheim and Woodward \cite[Theorem 6.3.1]{QuiltedHF} we transfer this knowledge to $L$.  This requires computing a certain induced relative spin structure, which we do by an indirect method, and the conditions on $\Char \K$ emerge.

We apply the same methods to the Lagrangian lens space to obtain

\begin{mthm}[\cref{P2P1Wide}, \cref{sscCP2}]
The Lagrangian $L(4,1)$ in $\C\P^2 \times \C\P^1$ is wide when $\Char \K$ is $7$ or $3$, depending on the choice of relative spin structure, and is narrow otherwise.
\end{mthm}


\subsection{Acknowledgements}

Much of this work was carried out whilst I was an EPSRC-funded PhD student in Cambridge, and I am grateful to my supervisor, Ivan Smith, for all of his help over my time there (one small part of which was the suggestion of applying the Gysin sequence to the examples in \cref{secGysin}).  I also thank Tim Perutz, Nick Sheridan, Brunella Torricelli, Chris Woodward, and Frol Zapolsky for helpful discussions.  I am supported by EPSRC grant [EP/P02095X/1].


\section{Floer theory prerequisites}
\label{secPrerequisites}


\subsection{Monotonicity}
\label{sscMonotonicity}

Recall that a Lagrangian $L \subset (X, \omega)$ is monotone if there exists a positive real number $\lambda_L$ such that $\omega = \lambda_L \mu$ as homomorphisms $\pi_2(X, L) \rightarrow \R$, where $\omega$ is the area and $\mu$ the Maslov index.  Similarly $X$ is monotone if there exists a positive $\lambda_X$ such that $\omega = 2\lambda_X c_1(X)$ as homomorphisms $\pi_2(X) \rightarrow \R$.  Assuming $\pi_2(X)\neq 0$, which will always be the case in our examples, monotonicity of $L$ implies montonicity of $X$ with $\lambda_L = \lambda_X$; we call this common value the \emph{monotonicity constant} of $X$ or $L$.  This is because $\mu$ restricts to $2c_1(X)$ on $\pi_2(X)$.  By the long exact sequence of the pair $(X, L)$, if $X$ is monotone and $\pi_1(L)$ is torsion then $L$ is also monotone.


\subsection{Relative spin structures}
\label{sscRelSpin}

A relative spin structure on a Lagrangian orients the moduli spaces of holomorphic discs which are counted by the Floer differential.  The mechanics of this do not concern us, but we will need to manipulate relative spin structures and how they change signs.

First we recap the basic definitions.  Suppose $V$ is an orientable real vector bundle $V$ over a space $M$.  We shall assume $\rank V \geq 2$, since ranks $0$ and $1$ require special treatment but are rather trivial.  The second Stiefel--Whitney class $w_2(V)$ is the obstruction to lifting the $\mathrm{GL}_+$-frame bundle of $V$ to a principal $\mathrm{GL}_+^\sim$-bundle, where $\mathrm{GL}_+$ is the orientation-preserving subgroup of $\mathrm{GL}(\rank V, \R)$ and $\mathrm{GL}_+^\sim$ is its unique connected double cover.  $V$ is spin if and only if $w_2(V)=0$, and in this case a spin structure is a choice of such a lift of the frame bundle.  The set of spin structures forms a torsor for $H^1(M; \Z/2)$.  If $M$ is an orientable manifold and $V=TM$ then we talk simply of $w_2(M)$ and spin structures on $M$.

\begin{defn}[{\cite[pp.~675--676]{FOOObig}, as reformulated in \cite[Remark 3.1.3, Proposition 3.1.5(b)]{QuiltOr}}]
An orientable submanifold $M$ of a manifold $N$ is \emph{relatively spin} if there exists a class $b$ in $H^2(N; \Z/2)$ with $b|_M = w_2(M)$.  In this case, a \emph{relative spin structure} on $M$ comprises a choice of $b$ (the background class) and an equivalence class of \v{C}ech $1$-cochain describing a principal $\mathrm{GL}_+^\sim$-bundle, with cocycle condition twisted by $b$, which lifts the frame bundle.  Here we are viewing the coefficient group $\Z/2$ of $b$ as the deck group of the covering $\mathrm{GL}_+^\sim \rightarrow \mathrm{GL}_+$.  The set of relative spin structures forms a torsor for $H^2(N, M; \Z/2)$.  For an orientable vector bundle $B$ on $N$, a relative spin structure with background class $w_2(B)$ is equivalent to a spin structure on $B|_M \oplus TM$.  We will usually be interested in the case where $N$ is symplectic and $M$ Lagrangian.
\end{defn}

The construction of orientations on holomorphic disc moduli spaces is described in \cite[Chapter 8]{FOOObig}.  All we need to know is how changing relative spin structure affects these orientations.

\begin{lem}[{de Silva \cite[Theorem Q]{VdS}, Cho \cite[Theorem 6.4]{ChoCl}, Fukaya--Oh--Ohta--Ono \cite[Proposition 8.1.16]{FOOObig}}]
\label{SpinChangeProp}
If two relative spin structures differ by a class $\eps \in H^2(X, L; \Z/2)$ then the associated orientations on the moduli space of discs in class $A \in H_2(X, L; \Z)$ differ by $(-1)^{\ip{\eps}{A}}$.\hfill$\qed$
\end{lem}

Similarly, the effect of a background class $b$ on quantum cohomology is to modify counts of rational curves in class $A \in H_2(X; \Z)$ by $(-1)^{\ip{b}{A}}$.


\subsection{Floer theory as a deformation of Morse theory}
\label{sscFloerMorse}

Given a monotone Lagrangian $L \subset X$ as in \cref{sscSetup}, its Floer theory can be regarded as a deformation of Morse theory in the following sense.  Biran--Cornea \cite{BCQS} construct a `pearl' model for $HF^*$ whose underlying chain complex is the Morse complex of a Morse function on $L$, tensored with the ring $\Lambda$.  The differential, product, and closed--open map only involve non-negative powers of $T$, so respect the filtration of the complex by $T$-exponent, and at the associated graded level (i.e.~the $T^0$ terms) they coincide with the corresponding classical operations: the Morse differential, the cup product, and the Morse restriction map $H^*(X) \rightarrow H^*(L)$ (recall that $H^*(X; \Lambda)$ is canonically identified with $H^*(X; \K) \otimes_\K \Lambda$ as a $\Lambda$-module).

The spectral sequence induced by this filtration is precisely the Oh spectral sequence, and the $r$th page differential encodes the $T^r$ terms.  In particular, it maps the surviving elements of $H^*(L; \K) \otimes \Lambda$ to $H^{*+1-rN_L}(L; \K) \otimes \Lambda$.  Combining grading considerations in this spectral sequence with the multiplicative structure, one obtains the following well-known result of Biran--Cornea.

\begin{prop}[{\cite[Proposition 6.1.1]{BCQS}}]\label{BCWideNarrow}
If $H^*(L; \K)$ is generated as an algebra by $H^{\leq m}(L; \K)$ with $m \leq N_L-1$ then $L^\flat$ is either wide or narrow, and only the former can occur if the inequality is strict.\hfill$\qed$
\end{prop}

Morse cocycles of index $< N_L$ are automatically Floer cocycles, and we obtain a `PSS' map $H^{< N_L}(L; \K) \rightarrow HF^*(L^\flat, L^\flat; \Lambda)$.  This intertwines $\CO^0$ with the restriction map $H^{< N_L}(X; \K) \rightarrow H^{<N_L}(L; \K)$


\section{Symplectic reduction}
\label{secReduction}

Our first observations relate to symplectic reduction.  After establishing some useful results about monotonicity and relative spin structures, including \cref{mlemRedMonotone}, we apply them to a family of Lagrangian flag varieties, and prove \cref{mHFFlagVariety}.


\subsection{Hamiltonian actions}

Recall that an action of a Lie group $K$ on $X$ is \emph{Hamiltonian}, with moment map $\mu : X \rightarrow \mathfrak{k}^*$ (not to be confused with the Maslov index!), if: the action preserves $\omega$; $\mu$ intertwines the $K$-action on $X$ with the coadjoint action on $\mathfrak{k}^*$; and $\mu$ generates the action in the sense that for all $\xi$ in $\mathfrak{k}$ we have
\[
\ip{\diff\mu}{\xi} = \omega(-, V_\xi),
\]
where $V_\xi$ is the vector field describing the action of $\xi$.  If $K$ is compact and acts freely on $\mu^{-1}(0)$, which implies that $0$ is a regular value of $\mu$, then the \emph{symplectic reduction} $X \sslash K$ is the quotient $\mu^{-1}(0)/K$ equipped with the unique symplectic form $\omega_{X \sslash K}$ whose pullback to $\mu^{-1}(0)$ coincides with $\omega|_{\mu^{-1}(0)}$.  If $L$ is a Lagrangian submanifold of $X$ contained in $\mu^{-1}(0)$, and is preserved setwise by the $K$-action, then $L/K$ defines a Lagrangian in $X \sslash K$.  We shall always assume that $K$ is connected, so in particular it acts in an orientation-preserving way.

We remark for later use that
\begin{lem}[{\cite[Proposition 1.3]{Ch}}]\label{OrbitsIsotropic}
$K$-orbits contained in $\mu^{-1}(0)$ are isotropic.\hfill$\qed$
\end{lem}

To do Floer theory with $L/K$ we would like to understand when it is monotone and relatively spin, and these are the subjects of the next two subsections.


\subsection{Monotonicity for reductions}

Our goal is to relate monotonicity of symplectic reductions to monotonicity upstairs:

\begin{prop}
\label{lemRedMonotone}
Suppose $X$ is a symplectic manifold carrying a Hamiltonian action of a compact connected Lie group $K$, which acts freely on $\mu^{-1}(0)$.  If $L \subset X$ is a $K$-invariant monotone Lagrangian submanifold contained in $\mu^{-1}(0)$ then $L/K$ is monotone in $X \sslash K$ with $\lambda_{L/K}=\lambda_L$.
\end{prop}
\begin{rmk}
Any connected $K$-invariant Lagrangian automatically lies in $\mu^{-1}(0)$, possibly after shifting $\mu$ by a fixed point of the coadjoint action---see \cite[Lemma 4.1]{CGS}.  The monotonicity constants appearing in the statement $\lambda_{L/K}=\lambda_L$ may not be uniquely determined, in which case we mean that there exist choices for which equality holds.
\end{rmk}
\begin{proof}
First we claim that $p_* : \pi_2(\mu^{-1}(0), L) \rightarrow \pi_2(X\sslash K, L/K)$ is an isomorphism, where $p$ denotes the quotient-by-$K$ map.  We will prove this by repeated application of the five lemma, so to begin let $Z$ denote $\mu^{-1}(0)$, let $K_0$ be a $K$-orbit inside $L$, and for $j\geq 1$ consider the following diagram of homotopy groups (or pointed sets when the degree is low)
\[
\begin{tikzcd}
\pi_j(K_0) \arrow{r}\arrow{d} & \pi_j(L) \arrow{r}\arrow{d} & \pi_j(L, K_0) \arrow{r}\arrow{d}{p_*} & \pi_{j-1}(K_0) \arrow{r}\arrow{d} & \pi_{j-1}(L) \arrow{d}
\\ \pi_j(K) \arrow{r} & \pi_j(L) \arrow{r} & \pi_j(L/K) \arrow{r} & \pi_{j-1}(K) \arrow{r} & \pi_{j-1}(L)
\end{tikzcd}
\]
The unlabelled vertical maps are the obvious isomorphisms, the top row is from the long exact sequence of the pair, and the bottom row is from the long exact sequence of the fibration.  Note that a priori $\pi_1(L, K_0)$ is only a pointed set, but the $K$-action on $L$ makes it into a group so that $p_*$ is a homomorphism; the multiplication on $K$ and connectedness of $L$ ensure that $\pi_0(K)=\pi_0(K_0)$ and $\pi_0(L)=\{\text{point}\}$ are also groups, but we don't need this.  The squares all commute (the only one which requires a little thought is the third of the four) so the five lemma shows that $p_* : \pi_j(L, K_0) \rightarrow \pi_j(L/K)$ is an isomorphism for all $j \geq 1$.  Similarly we have isomorphisms $p_* : \pi_j(Z, K_0) \rightarrow \pi_j(Z/K = X \sslash K)$ for all $j \geq 1$.

Now consider the commutative diagram of homotopy groups (these really are groups, as above)
\[
\begin{tikzcd}
\pi_2(L, K_0) \arrow{r}\arrow{d}{p_*} & \pi_2(Z, K_0) \arrow{r}\arrow{d}{p_*} & \pi_2(Z, L) \arrow{r}\arrow{d}{p_*} & \pi_1(L, K_0) \arrow{r}\arrow{d}{p_*} & \pi_1(Z, K_0) \arrow{d}{p_*}
\\ \pi_2(L/K) \arrow{r} & \pi_2(X\sslash K) \arrow{r} & \pi_2(X\sslash K, L/K) \arrow{r} & \pi_1(L/K) \arrow{r} & \pi_1(X\sslash K)
\end{tikzcd}
\]
The top row is from the long exact sequence of the triple $(Z, L, K_0)$ whilst the bottom row is from the long exact sequence of the pair.  By the previous paragraph the vertical maps are all isomorphisms, except possibly the middle one, so applying the five lemma again we see that the middle map $p_* : \pi_2(Z, L) \rightarrow \pi_2(X\sslash K, L/K)$ is also an isomorphism, as claimed.

Our strategy now is to understand monotonicity of $L/K$ by lifting discs by $(p_*)^{-1}$, so take an arbitrary class $\beta$ in $\pi_2(X \sslash K, L/K)$ and choose a disc $u : (D, \pd D) \rightarrow (Z, L)$ with $[p \circ u] = \beta$.  Since $p^* \omega_{X \sslash K} = \omega_X |_{Z}$ we immediately have that $u$ and $\beta$ have equal areas.  We claim that they have equal Maslov indices, i.e.~that the bundle pairs $(u^*TX, u|_{\pd D}^* TL)$ and $(u^*p^*T(X\sslash K), u|_{\pd D}^* p^* T(L/K))$ have equal indices, and we shall do this by exhibiting the latter pair as a quotient of the former by a trivial sub-pair.

So fix an $\omega$-compatible almost complex structure $J$ on $X$.  For any non-zero $\xi \in \mathfrak{k}$ we have
\[
\ip{(JV_\xi) \lrcorner \diff \mu}{\xi} = \omega(JV_\xi, V_\xi) > 0,
\]
and hence $(JV_\xi) \lrcorner \diff \mu \neq 0$, so the subbundle $J(\mathfrak{k} \cdot Z)$ of $TX|_{Z}$ is sent fibrewise injectively to $\mathfrak{k}^*$ by $\diff \mu$ and thus provides a complementary subbundle to $TZ$.  This means that $\mathfrak{k} \cdot Z$ is a subbundle of $TZ$ which is complementary to $E_Z \coloneqq TZ \cap J(TZ)$, giving
\[
TX|_{Z} = (\mathfrak{k} \cdot Z) \oplus J(\mathfrak{k} \cdot Z) \oplus E_Z
\]
as real vector bundles over $Z$.  Letting $E_\mathfrak{k}$ denote the sum of the first two terms, we obtain a splitting $TX|_{Z} = E_\mathfrak{k} \oplus E_Z$ as \emph{complex} vector bundles.

Now let $F_\mathfrak{k}$ denote the totally real subbundle $\mathfrak{k} \cdot L = TL \cap E_\mathfrak{k}|_L$ of $E_\mathfrak{k}|_L$, and $F_Z$ the projection of $TL$ onto $E_Z|_L$ with kernel $F_\mathfrak{k}$.  The short exact sequence of pairs
\[
0 \rightarrow (E_{\mathfrak{k}}, F_\mathfrak{k}) \rightarrow (TX|_Z, TL) \rightarrow (E_Z, F_Z) \rightarrow 0
\]
then gives
\[
\mu(u^*TX, u|_{\pd D}^* TL) = \mu(u^*E_\mathfrak{k}, u|_{\pd D}^*F_\mathfrak{k}) + \mu(u^*E_Z, u|_{\pd D}^*F_Z).
\]
The first term on the right-hand side vanishes, since the bundle pair is trivialised by the action of $\mathfrak{k}$, and we are left to show that the second term is $\mu(u^*p^*T(X\sslash K), u|_{\pd D}^*p^*T(L/K))$.  To see that this is indeed the case, note that $E_Z$ projects isomorphically onto $p^*T(X\sslash K) = TX|_{Z} / E_\mathfrak{k}$ and that this projection identifies $F_Z$ with $p^*T(L/K)$.  Moreover, the complex structure on $E_Z$ is compatible with $p^*\omega_{X\sslash K}=\omega|_{TZ}$.  We conclude that $u$ and $\beta$ have equal indices and so if $L$ is monotone then $L/K$ is monotone with the same monotonicity constant.
\end{proof}


\subsection{Relative spin structures for reductions}

We would like to understand relative spin structures on reduced Lagrangians $L/K \subset X \sslash K$ in the setting of \cref{lemRedMonotone}, assuming also that $L$ is orientable.  These often (for example, if the $K$-action extends to a larger group $K'$ and $L$ is the zero set of the $K'$-moment map; see \cref{FlagRelSpin}) have the property that their normal bundle extends to $X \sslash K$, and in this case we can apply

\begin{lem}
\label{lemRedRelSpin}
Given a vector bundle $V$ on $N$ whose restriction to $M$ is identified with the normal bundle to $M$ in $N$, $M$ carries a natural relative spin structure with background class $w_2(V)+w_2(N)$.
\end{lem}
\begin{proof}
We need to show that $TM \oplus V|_M \oplus TN|_M$, i.e.~$(TN \oplus TN)|_M$, carries a natural spin structure.  But the double of any orientable rank $k$ vector bundle has a ntural spin structure since its structure group naturally reduces to the block diagonal subgroup $\mathrm{GL}(k, \R)_+ \hookrightarrow \mathrm{GL}(2k, \R)_+$, and this lifts to $\mathrm{GL}(2k, \R)_+^\sim$ \cite[Proposition 3.1.6(b)]{QuiltOr}.
\end{proof}

To apply this one needs to compute $w_2(V)$, and taking $M=L/K \subset N = X \sslash K$ and assuming $V$ pulls back to a spin bundle on $Z=\mu^{-1}(0)$, this can be done as follows.  Fix a spin structure on $p^*V$, where $p : Z \rightarrow X \sslash K$ is the projection, and for a loop $\gamma \in \pi_1(K)$ define $\eps(\gamma) \in \Z/2$ to be $0$ if the action of $\gamma$ on $p^*V$ lifts to the $\mathrm{GL}_+^\sim$-bundle determined by the spin structure, and $0$ otherwise.  Doing this for all classes in $\pi_1(K)$ we obtain an element $\eps$ of $\Hom(\pi_1(K), \Z/2)=H^1(K; \Z/2)$.

\begin{lem}
\label{wtwoV}
The class $\eps \in H^1(K; \Z/2)$ transgresses to $w_2(V) \in H^2(X\sslash K; \Z/2)$ in the Serre spectral sequence for $K \hookrightarrow Z \twoheadrightarrow X\sslash K$.
\end{lem}
\begin{proof}
Consider the two fibrations $K \hookrightarrow Z \twoheadrightarrow X\sslash K$ and $B(\Z/2) \hookrightarrow B\mathrm{GL}_+^\sim \twoheadrightarrow B\mathrm{GL}_+$.  The bundle $V$ is classified by a map $\phi : X\sslash K \rightarrow B\mathrm{GL}_+$ between their base spaces, and the choice of spin structure on $p^*V$ gives a lift $\widetilde{\phi}$ to the total spaces.  Restricting to the fibres yields a map $K \rightarrow B(\Z/2)$, i.e.~a class $\eps'$ in $H^1(K; \Z/2)$, and we claim that $\eps' = \eps$.  To see this, it suffices to prove that looped maps
\[
\Omega \eps, \Omega \eps' : \Omega K \rightarrow \Z/2
\]
are homotopic as morphisms of $A_\infty$-spaces.  But both maps can be described as the restriction to $\Omega K$ of the map $\Omega X \rightarrow \mathrm{GL}_+^\sim$ describing the monodromy of our bundle with respect to some connection.

Finally we need to show that $\eps'$ transgresses to $w_2(V)$.  For this we deloop the fibrations to obtain a diagram up to homotopy
\[
\begin{tikzcd}
Z \arrow[hook]{r}\arrow{d}{\widetilde{\phi}} & X\sslash K \arrow[twoheadrightarrow]{r}{\pi}\arrow{d}{\phi} & BK \arrow{d}{B\eps'}
\\ B\mathrm{GL}_+^\sim \arrow[hook]{r} & B\mathrm{GL}_+ \arrow[twoheadrightarrow]{r}{w_2} & B^2\Z/2
\end{tikzcd}
\]
By definition, $w_2(V)$ is the composite $w_2 \circ \phi = B\eps' \circ \pi$, where $\pi$ is the classifying map $X/K \rightarrow BK$ as shown.  The class $\eps'$ transgresses to $B\eps'$ in the Serre spectral sequence for $K \hookrightarrow EK \twoheadrightarrow BK$.  Pulling back by the map $\pi$ gives our fibration $K \hookrightarrow X \twoheadrightarrow X \sslash K$, so by naturality of the Serre spectral sequence we see that $\eps'$ transgresses to $w_2(V)$ in the spectral sequence for the latter.
\end{proof}

\begin{rmk}
Changing the spin structure on $p^*V$ by a class $\delta$ in $H^1(Z; \Z/2)$ changes $\eps$ by the pullback of $\delta$ to $K$.  However, the image of this pullback is precisely the kernel of the map
\[
\pi^* : H^2(BK; \Z/2) \cong H^1(K; \Z/2) \rightarrow H^2(X\sslash K; \Z/2),
\]
so the transgression $w_2(V)$ is unaffected.
\end{rmk}


\subsection{Worked example: flag varieties}
\label{sscLagrangianFlagVariety}

Consider the space $X=\C^{n^2}$ equipped with the standard symplectic form.  Fix a tuple of positive integers $k_1, \dots, k_r$ which sum to $n$ (with $r \geq 2$), and let $K$ be the block diagonal subgroup $\U(k_1)\times \dots \times \U(k_r)$ of $\U(n)$.  We shall define a Hamiltonian $K$-action on $X$ and a Lagrangian $L$ in the zero set of the moment map, apply the above results to show that $L/K$ has well-defined Floer theory in $X\sslash K$, and use the closed--open map to compute $HF^*(L/K, L/K)$ as a ring from knowledge of the quantum cohomology of $X\sslash K$.

The $K$-action will actually be defined as the restriction of a Hamiltonian $\U(n)$-action, so we discuss that first.  To construct this action we view elements of $\C^{n^2}$ as $n\times n$ matrices $w$, so that $\U(n)$ acts by left multiplication.  This action is Hamiltonian with moment map $\mu_{\U(n)} : X \rightarrow \mathfrak{u}(n)^*$ given by
\begin{equation}
\label{UnMomentMap}
\ip{\mu_{\U(n)}(w)}{A} = \frac{i}{2}\Tr(A) -\frac{i}{2} \Tr (w^\dag A w) = \frac{1}{2}\ip{i(ww^\dag-I)}{A}
\end{equation}
for all $w \in \C^{n^2}$ and $A \in \mathfrak{u}(n)$, where $I$ is the $n\times n$ identity matrix and $\ip{A_1}{A_2}=\Tr(A_1^\dag A_2)$ is the usual inner product on $\mathfrak{u}(n)$ (note that $iww^\dag$ and $iI$ both lie in this space); the $\Tr(A)$ term controls the normalisation of the reduced symplectic form.  Therefore $\mu_{\U(n)}^{-1}(0)$ comprises those $w$ such that $ww^\dag=I$, so is exactly $\U(n) \subset \C^{n^2}$.  By \cref{OrbitsIsotropic} it is isotropic, and its dimension is $n^2 = \dim_\C X$, so it is Lagrangian.  This will be our $L$.

\begin{lem}
\label{lemUnMonotone}
$L$ is orientable and monotone.
\end{lem}
\begin{proof}
$L$ is actually parallelisable since it's a Lie group.  To prove monotonicity consider the holomorphic disc $u(z)$ in $X$ with boundary on $L$ given by the diagonal matrix with diagonal entries $z, 1, \dots, 1$.  The boundary of $u$ generates $\pi_1(L)$, so from the long exact sequence $u$ generates $\pi_2(X,L)$.  It has index $2>0$ and area $\pi > 0$.
\end{proof}

Now restrict this action to $K$.  The resulting $K$-action is Hamiltonian and its moment map $\mu$ is given by projecting $\mu_{\U(n)}$ under
\[
\mathfrak{u}(n)^* \rightarrow \mathfrak{k}^* = \big(\mathfrak{u}(k_1) \oplus \dots \oplus \mathfrak{u}(k_r)\big)^*.
\]
This is the Hamiltonian action we shall study.  The space $Z=\mu^{-1}(0)$ comprises those $w$ such that $i(ww^\dag-I)$ is orthogonal to $\mathfrak{k}$, which amounts to each row having norm $1$, the first $k_1$ rows being pairwise orthogonal, and similarly for the next $k_2$, and so on.  Note that $K$ acts freely on $Z$, and $L$ lies in $Z$ and is $K$-invariant, so we are in the setting of \cref{lemRedMonotone}, which tells us that

\begin{lem}
$L/K$ is monotone.\hfill$\qed$
\end{lem}

Our goal is the following result

\begin{thm}
\label{HFFlagVariety}
Equipping $L/K$ with a specific choice of relative spin structure (defined below) and with the trivial line bundle to give a brane $(L/K)^\flat$, we have
\[
HF^*((L/K)^\flat, (L/K)^\flat; \Lambda) = \Lambda[c_{1,1}, \dots, c_{1,k_1}, c_{2,1}, \dots, c_{r, k_r}]\bigg/\bigqcup_{j=1}^r (1+c_{j,1}+\dots+c_{j,k_j}) = 1+T,
\]
where $\qcup$ is the Floer product, each $c_{j,k}$ has degree $2k$, and the Novikov variable $T$ has degree $2n$.
\end{thm}

\begin{rmk}
We'll see shortly that $L/K$ is simply connected so in fact the only flat line bundle it admits is the trivial one.
\end{rmk}

The first step is to understand the topology of $L/K$.  Letting $V_j(w) \subset \C^n$ denote the span of the first $d_j \coloneqq k_1+\dots+k_j$ rows of a matrix $w$ in $L$, we obtain a diffeomorphism
\[
L/K \rightarrow \text{flag variety } F(d_1, \dots, d_r)
\]
given by
\[
w \mapsto 0 = V_0(w) \subset V_1(w) \subset \dots \subset V_r(w) = \C^n.
\]
We have tautological bundles $E_1, \dots, E_r$ of ranks $k_1, \dots, k_r$, with fibres $V_r(w)/V_{r-1}(w)$ respectively.

\begin{lem}
\label{lemLKtopology}
$L/K=F(d_1, \dots, d_r)$ is simply connected and $H^*(L/K; \Z)$ is the polynomial algebra by the Chern classes of the tautological bundles $E_1, \dots, E_r$ modulo the relations
\[
c(E_1) \smile \dots \smile c(E_r) = 1
\]
arising from the fact that their sum is trivial.
\end{lem}
\begin{proof}
The first claim follows from the long exact sequence in homotopy groups for the fibration
\[
K = \U(k_1)\times\dots\times \U(k_r) \hookrightarrow L= \U(n) \twoheadrightarrow L/K=F(d_1, \dots, d_r).
\]
The second, meanwhile, comes from the Serre spectral sequence for the fibration
\[
\U(n) \hookrightarrow \U(n) \times_K EK \twoheadrightarrow BK,
\]
whose total space is homotopy equivalent to $L/K$.  Explicitly, the cohomology of the base is the polynomial algebra on the Chern classes of the $E_j$, whilst the generators $x_1, x_3, \dots, x_{2n-1}$ for the cohomology of the fibre transgress to the Chern classes of $E_1 \oplus \dots \oplus E_r$.
\end{proof}

The symplectic reduction $X \sslash K$ has a similar description as the product
\[
\Gr(k_1, n) \times \dots \times \Gr(k_r, n)
\]
of Grassmannians.  The cohomology of each $\Gr(k_j, n)$ is the polynomial algebra on the Chern classes of its tautological bundle $E_j$ (of rank $k_j$) and of the quotient $F_j=\underline{\C}^n/E_j$ (of rank $n-k_j$), modulo the relation $c(E_j) \smile c(F_j)=1$.  These bundles $E_j$ restrict to the $E_j$ above on $L/K$, so our use of the same notation is justified; if we wish to distinguish them we will denote them by $E_j(X\sslash K)$ and $E_j(L/K)$.  The tangent bundle $T\Gr(k_j, n)$ is naturally identified with $E_j^\vee \otimes F_j$---the fibre of the latter over the point of $\Gr(k_j, n)$ corresponding to a rank $k_j$ subspace $V \subset \C^n$ comprises linear maps $V \rightarrow \C^n/V$, and hence infinitesimal deformations of $V$---so by considering Chern roots we obtain $c_1(\Gr(k_j, n))=k_jc_1(F_j)-(n-k_j)c_1(E_j)$.  We have just seen that $H^2(\Gr(k_j, n); \Z)$ is free of rank $1$, generated by $c_1(E_j)=-c_1(F_j)$, so we deduce $c_1(\Gr(k_j, n))=-nc_1(E_j)$.

\begin{lem}
\label{FlagMinimalMaslov}
The minimal Maslov number $N_{L/K}$ is $2n$.
\end{lem}
\begin{proof}
Since $L/K$ is simply connected (by \cref{lemLKtopology}) any disc in $\pi_2(X\sslash K, L/K)$ lifts to a sphere in $\pi_2(X\sslash K)$, and the Maslov index becomes twice the Chern number.  We just saw that $c_1(X\sslash K) = -n \sum_j c_1(E_j)$, so it suffices to show that there is a sphere which pairs to $\pm 1$ with $c_1(E_j)$.  To do this, pick $k_j+1$ linearly independent vectors $v_1, \dots, v_{k_j+1}$ in $\C^n$ and consider the sphere
\[
u:[1:z] \in \C\P^1 \mapsto \lspan{v_1, \dots, v_{k_j-1}, v_{k_j}+zv_{k_j+1}} \in \Gr(k_j, n).
\]
The bundle $u^*E_j$ is $\mathcal{O}_{\C\P^1}^{\oplus k_j-1} \oplus \mathcal{O}_{\C\P^1}(-1)$, so $\ip{[u]}{c_1(E_j)}=-1$, as needed.
\end{proof}

\begin{lem}
\label{FlagRelSpin}
$L/K$ carries a natural relative spin structure with background class $b \coloneqq \sum_j k_jc_1(E_j)$.
\end{lem}
\begin{proof}
To construct the relative spin structure we apply \cref{lemRedRelSpin}.  The normal bundle to $L$ in $X$ is naturally identified with the trivial bundle with fibre $\mathfrak{u}(n)^*$, whilst the normal bundle to $Z$ is identified with the trivial bundle with fibre $\mathfrak{k}^*$.  Therefore the normal bundle to $L$ in $Z$ is the trivial bundle with fibre $\Phi = (\mathfrak{u}(n)/\mathfrak{k})^*$.  This extends $K$-equivariantly to all of $Z$ (the $K$-action is coadjoint) and thus descends to a vector bundle $V$ on $X\sslash K$.  By \cref{lemRedRelSpin} $L/K$ carries a natural relative spin structure with background class $w_2(X\sslash K)+w_2(V)$.

We compute $w_2(V)$ as in \cref{wtwoV}, equipping $p^*V$ with the spin structure induced by its trivialisation as $Z \times \Phi$.  The class $\eps$ is precisely the pushforward on $\pi_1$ given by the coadjoint action map $K \rightarrow \mathrm{GL}_+$, and $\pi_1(K)$ is freely generated by the loops $\gamma_1, \dots, \gamma_r$ defined as follows: $\gamma_j(t)$ is diagonal with diagonal entries $1, \dots, e^{it}, \dots, 1$, where $e^{it}$ occurs in the $d_j$th position.  The coadjoint action of $\gamma_j$ loops $n-k_j$ times the generator of $\pi_1(\mathrm{GL}_+)$, so we conclude that $\eps$ transgresses to
\[
\sum_{j=1}^r -(n-k_j)c_1(E_j).
\]
Hence $w_2(X\sslash K)+w_2(V) = -n\sum_j c_1(E_j) - \sum_j(n-k_j)c_1(E_j) = \sum_j k_j c_1(E_j)$ in $H^2(X\sslash K; \Z/2)$.
\end{proof}

We are almost ready to prove \cref{HFFlagVariety}.  The final input we need is

\begin{lem}
The quantum cohomology $QH^*(X \sslash K, b; \Lambda)$ is the polynomial algebra over $\Lambda$ on the Chern classes of the $E_j$ and $F_j$, modulo the relations $c(E_j) \qcup c(F_j) = 1+T$, where $T$ is the Novikov variable of degree $2n$.
\end{lem}
\begin{proof}
We already saw that without the $T$ term this gives the classical cohomology.  Witten \cite[Section 3.2]{WittenGrassmannian} showed that with zero background class the quantum correction is $(-1)^{k_j}T$ (Witten's relation is in terms of the duals of these bundles so comes with a different sign), and we are left to show that turning on the class $b$ modifies this sign to $+1$.  To prove this, recall from \cref{sscRelSpin} that a background class $b$ changes the count of curves in class $A$ by a factor of $(-1)^{\ip{b}{A}}$.  In our case, the curves contributing to $c(E_j)\qcup c(F_j)$ lie on the $\Gr(k_j, n)$ factor and have Chern number $n$ (for degree reasons) and thus pair to $-1$ with $c_1(E_j)$ and to zero with all other $c_1(E_k)$.  They therefore pair to $k_j \mod 2$ with $b$, giving precisely the required factor.
\end{proof}

\begin{rmk}
If $k=1$ then $\Gr(k,n)$ is $\C\P^{n-1}$ and $E$ is $\mathcal{O}_{\C\P^{n-1}}(-1)$ with Chern class $1-H$, where $H$ is the hyperplane class.  The relation $c(E)\smile c(F)=1$ tells us that $c(F)=1+H+\dots+H^{n-1}$, then Witten's relation reduces to $H^{\qcup n}=T$: the familiar description of $QH^*(\C\P^{n-1})$.
\end{rmk}

\begin{proof}[Proof of \cref{HFFlagVariety}]
The Chern classes of the $E_j(L/K)$ lie in degree $<N_L$ (by \cref{FlagMinimalMaslov}) and therefore define Floer cohomology classes $c_{j,k}$ via the PSS map.  Since the classical versions generate $H^*(L/K; \K)$ as a $\K$-algebra with respect to the cup product (\cref{lemLKtopology}), the Floer versions generate $HF^*=HF^*((L/K)^\flat, (L/K)^\flat; \Lambda)$ as a $\Lambda$-algebra with respect to the Floer product.

We also have the Chern classes of the $E_j(X\sslash K)$ in $QH^*=QH^*(X\sslash K, b; \Lambda)$, and using again the fact that they lie in degree $<N_L$ we see that their images under $\CO$ coincide with their classical restrictions to $L/K$ via PSS.  In other words, for each $j$ we have
\[
\CO(c(E_j(X\sslash K))) = c(E_j(L/K)) \in HF^*.
\]
This also forces the $c_{j,k}$ to commute with respect to the Floer product, since $QH^*$ is commutative.

For each $j$, the complement $F_j$ restricts to
\[
E_1(L/K) \oplus \dots \oplus \widehat{E_j(L/K)} \oplus \dots \oplus E_r(L/K)
\]
on $L/K$, so its total Chern class (which lies in degree $<N_L$) satisfies
\[
\CO(c(F_j)) = \bigqcup_{l\neq j} (1+c_{l,1}+\dots+c_{l,k_l}).
\]
The relation $c(E_j) \qcup c(F_j) = 1+T$ in $QH^*$, and the fact that $\CO$ is a ring map, then gives
\begin{equation}
\label{HFFlagRelation}
\bigqcup_{j=1}^r (1+c_{j,1}+\dots+c_{j,k_j}) = 1+T
\end{equation}
in $HF^*$.  Therefore $HF^*$ is a quotient of the algebra claimed in \cref{HFFlagVariety}.  To show that this quotient is by zero, and thus prove the theorem, it suffices to show that $HF^*$ and $A$ have the same dimension over $\K$ in each degree, i.e.~that $(L/K)^\flat$ is wide, and this follows from \cref{BCWideNarrow}.
\end{proof}

We end by noting that if we had tried to run this argument without the background class $b$ then the right-hand side of \eqref{HFFlagRelation} would have been $1+(-1)^{k_j}T$, and varying $j$ we would have obtained relations which are inconsistent outside characteristic $2$ unless the $k_j$ all have the same parity.  The most general background class that makes them consistent is of the form $b + \delta \sum_j c_1(E_j)$ for $\delta \in \Z/2$, which makes the right-hand side into $1+(-1)^\delta T$, and these two choices correspond to the two different relative spin structures on $L/K$: recall that relative spin structures form a torsor for $H^2(X\sslash K, L/K; \Z/2)$, and in our case the long exact sequence of the pair tells us that this group is precisely the kernel of $H^2(X\sslash K; \Z/2) \rightarrow H^2(L/K; \Z/2)$ since $L/K$ is simply connected; this kernel is the span of $\sum_j c_1(E_j)$, so there is a unique relative spin structure which differs from that in \cref{FlagRelSpin}, and the two background classes differ by $\sum_j c_1(E_j)$.


\section{The Floer--Poincar\'e polynomial}
\label{secFloerPoincare}

In this section we explore consequences of grading periodicity.


\subsection{Periodicity and the Poincar\'e polynomial}

So far we have worked with the $\Z$-graded $\Lambda$-algebra $HF^*(L^\flat, L^\flat; \Lambda)$, but by setting the Novikov variable $T$ to $1$ we can turn it into a $\Z/N_L$-graded $\K$-algebra which we denote by $HF^*(L^\flat, L^\flat; \K)$.

\begin{defn}
For an integer $q$, say $L^\flat$ is \emph{$q$-periodic} if $HF^*(L^\flat, L^\flat; \K) \cong HF^{*+q}(L^\flat, L^\flat; \K)$ as $\Z/N_L$-graded vector spaces.
\end{defn}

$L^\flat$ is tautologically $N_L$-periodic, but it often turns out to be $q$-periodic for some proper divisor $q$ of $N_L$, and this can impose strong restrictions (see for example the work of Seidel \cite{SeidelGr} and Biran--Cornea \cite{BCQS, BCRU}).  We shall introduce a simple tool, closely related to the discrete Fourier transform, which allows us to extract new information.

\begin{rmk}
Sources of $q$-periodicity include:
\begin{itemize}
\item The quantum module action of invertible elements of degree $q$ in the $\Z/N_L$-graded quantum cohomology of $QH^*(X; \K)$ (obtained by setting $T=1$ in $QH^*(X; \Lambda)$), as in \cite[Corollary 6.2.1]{BCQS}.  For example, monotone Lagrangians in compact toric varieties are $2$-periodic since the toric divisiors are quantum invertible (McDuff--Tolman \cite[Section 5.1]{McDuffTolman} exhibit them as elements in the image of the Seidel representation \cite{SeidelElt}
\[
\pi_1(\mathrm{Ham}(X)) \rightarrow QH^*(X; \K)^\times,
\]
which in general provides a rich source of invertibles).  Similarly, monotone Lagrangians in quadrics of even complex dimension are $2$-periodic by \cite[Lemma 4.3]{SmQuadrics}.  Note that if our Lagrangian is equipped with a relative spin structure with background class $b$ then we should work with $QH^*(X, b; \K)$.
\item An isomorphism between the shift functor $[q]$ and the identity functor on the Fukaya category of $X$, as in \cite[Section 3]{SeidelGr}.  For example, on $\C\P^n$ the path of symplectomorphisms
\[
(\phi_t : [z_0: z_1: \dots : z_n] \mapsto [e^{2\pi i t}z_0: z_1: \dots : z_n])_{t \in [0, 1]}
\]
gives a Hamiltonian isotopy between the identity and $[-2]$.  In fact, the Seidel representation sends this path (which is actually a closed loop) to the hyperplane class, which provides $2$-periodicity via the quantum module action.
\item  The Floer--Gysin sequence of Biran--Khanevsky \cite[Corollary 1.3]{BirKhan}, in the case $q=2$ (and $\Char \K = 2$, although this restriction can probably be lifted).  Here one assumes that $X$ embeds as a codimension $2$ symplectic submanifold of some symplectic manifold $(M, \omega_M)$, so that $X$ is Poincar\'e dual to a positive multiple of $[\omega_M]$ and $M \setminus X$ is subcritical.  A monotone Lagrangian in $X$ lifts to a Lagrangian circle bundle in $M \setminus X$, and the self-Floer cohomology of the latter is the cone over `multiplication by the Euler class' on the self-Floer cohomology of the former.  Subcriticality ensures that the circle bundle must have vanishing self-Floer cohomology, so multiplication by the Euler class is an isomorphism of degree $2$.\qedhere
\end{itemize}
\end{rmk}

The main idea is to encode the degreewise dimensions of $HF^*$ in a generating function.  Recall that the Poincar\'e polynomial of $L$ over $\K$ is
\[
P(S) = \sum_{j \in \Z} \dim_\K H^j(L; \K)  \cdot S^j \in \Z[S].
\]
This is multiplicative under tensor product decompositions of $H^*(L; \K)$ as a graded vector space.

\begin{defn}
The \emph{Floer--Poincar\'e polynomial} of $L^\flat$ is the generating function
\[
P_F(S) = \sum_{j \in \Z/N_L} \dim_\K HF^j(L^\flat, L^\flat; \K) \cdot S^j \in \Z[S]/(S^{N_L} - 1)
\]
of the `Floer--Betti numbers'.  Note that this is only defined modulo $S^{N_L}-1$.
\end{defn}

Its key property for us is

\begin{prop}
\label{lemPeriodic}
If $L^\flat$ is $q$-periodic for some proper divisor $q$ of $N_L$, then $P_F(S)$ is divisible by $S^q+S^{2q}+\dots+S^{N_L}$.  In particular, the total dimension $P_F(1)$ of $HF^*(L^\flat, L^\flat; \K)$ is divisible by $N_L/q$, whilst $P_F(\zeta)$ must vanish for any $N_L$th root of unity $\zeta$ that isn't also a $q$th root of unity.
\end{prop}
\begin{proof}
Assuming $q$-periodicity we have
\begin{align*}
P_F(S) &= \sum_{j =1}^q \sum_{k=1}^{N_L/q} \dim_\K HF^{j+kq}(L^\flat, L^\flat; \K) \cdot S^{j+kq}
\\ &= \sum_{j =1}^q \dim_\K HF^j(L^\flat, L^\flat; \K) \cdot S^j \sum_{k=1}^{N_L/q} S^{kq},
\end{align*}
and the inner sum is the claimed factor.  When $S=1$ this factor becomes $N_L/q$.  Meanwhile, when $S^q \neq 1$ it becomes $S^q(S^{N_L}-1)/(S^q-1)$, so vanishes if $S$ is also an $N_L$th root of unity.
\end{proof}

\begin{rmk}
For $\zeta = e^{-2\pi i /N_L}$ the sequence $P_F(1), P_F(\zeta), \dots, P_F(\zeta^{N_L-1})$ is the discrete Fourier transform of the sequence of Floer--Betti numbers.
\end{rmk}


\subsection{Minimal Maslov constraints}

A simple consequence for monotone tori is

\begin{prop}\label{corTorMinMas}  A monotone Lagrangian torus $L \subset X$ which is $2$-periodic has minimal Maslov number $2$.  In particular this applies to any monotone Lagrangian torus if the $\Z/2N_X$-graded quantum cohomology contains an invertible element of degree $2$.
\end{prop}
\begin{proof}  Suppose for contradiction that $L$ is a $2$-periodic monotone Lagrangian $n$-torus of minimal Maslov number $N_L > 2$, and equip it with an arbitrary spin structure and flat line bundle over $\K$.  By \cref{BCWideNarrow} $L^\flat$ is wide, so its Floer--Poincar\'e polynomial is the reduction mod $(S^{N_L}-1)$ of its ordinary Poincar\'e polynomial, $(1+S)^n$.  Applying $P_F(\zeta) = 0$, where $\zeta = e^{2\pi i/N_L}$, we see that $1 + \zeta = 0$.  Since $N_L> 2$ this is impossible, so we conclude that such an $L$ cannot exist.
\end{proof}

\begin{rmk}
Using his theory of Floer (co)homology on the universal cover, Damian \cite[Theorem 1.6]{DamianUnivCov} proved that $N_L = 2$ for all monotone orientable aspherical Lagrangians $L$ in the product of $\C\P^n$ (for $n \geq 1$) with an arbitrary symplectic manifold $W$.  Fukaya \cite[Theorem 14.1]{FukayaAppHF} obtained a similar result, without monotonicity, for aspherical spin Lagrangians in $\C\P^n$ and other uniruled symplectic manifolds.
\end{rmk}

In fact, the same argument generalises to give

\begin{thm}
\label{thmExteriorAlgebraNL}
Suppose $L \subset X$ is a (closed, connected) monotone Lagrangian, whose cohomology $H^*(L; \K)$ is an exterior algebra on generators of degree $2k_1-1 \leq \dots \leq 2k_r-r$, and that either $\Char \K=2$ or that $L$ is orientable and relatively spin (with background class $b$).  If $QH^*(X, b; \K)$ contains an invertible element of even degree $q\leq 2k_r$ then $N_L$ is at most $2k_r$.
\end{thm}
\begin{proof}
Let $L^\flat$ be the brane obtained by equipping $L$ with an arbitrary relative spin structure (if $\Char \K \neq 2$) and the trivial line bundle.  If $N_L > 2k_r$ then in particular $N_L \geq 2$ so $HF^*(L^\flat, L^\flat; \K)$ is defined, and from \cref{BCWideNarrow} we get
\[
P_F(S)=\prod_{j=1}^r (1+S^{2k_j-1}) \mod (S^{N_L}-1).
\]
We claim that this is impossible, assuming the existence of the invertible element $h$ of degree $q$.

Well, after reducing the grading of $QH^*$ modulo $N_L$, some power of $h$ constitutes an invertible element of degree $q'$, where $q' \coloneqq \gcd(q, N_L)$ is the smallest element of the subgroup of $\Z/N_L$ generated by $q$.  \Cref{lemPeriodic} then gives $P_F(\zeta)=0$ for $\zeta=e^{2\pi i/N_L}$, and $N_L/q' | 2^r$.  The former means that $1+\zeta^{2k_j-1}=0$ for some $j$, and hence that $N_L=2(2k_j-1)$, and the latter then yields $q'=2k_j-1$ or $2(2k_j-1)$.  Since $N_L$ is even, our assumption that $q$ is even means that $q'$ is also even.  We must therefore have $q'=2(2k_j-1)=N_L$, but this contradicts the fact that $q'\leq q\leq 2k_r < N_L$.  We conclude that if the element $h$ exists then $N_L$ must be at most $2k_r$.
\end{proof}

\begin{rmk}
The hypotheses are satisfied for monotone Lagrangian embeddings of compact connected Lie groups $K$ and their quotients by finite subgroups $\Gamma$, taking $\K=\Q$.  Hopf \cite[Satz 1]{HopfUberDieTopologie} showed that the rational cohomology algebras of such $K$ are exterior algebras on odd degree generators, and the same holds for finite quotients by showing that the quotient map induces an isomorphism on rational homology (an inverse is provided by sending a simplex in $K/\Gamma$ to the average of its $|\Gamma|$ lifts to $K$: this is clearly a right inverse, and to see that it's a left inverse note that if we project a cycle $\sigma$ in $K$ and then lift we obtain
\[
\frac{1}{|\Gamma|} \sum_{\gamma \in \Gamma} \gamma \cdot \sigma
\]
and this is homotopic to $\sigma$ by homotoping each $\gamma$ back to the identity).  All such $K/\Gamma$ are orientable and spin since they are parallelised by the infinitesimal action of $\mathfrak{k}$.
\end{rmk}


\subsection{Worked example: $\PSU(n)$}
\label{sscLCEL}

We now apply these ideas to the following family of examples.  For an integer $n \geq 2$ the group $\SU(n)$ acts by left multiplication on the space of $n\times n$ complex matrices, and projectivising gives a Hamiltonian $\PSU(n)$-action on $X = \C\P^{n^2-1}$ (with the Fubini--Study form).  The action on the identity matrix is free, and its orbit $L$ gives a Lagrangian embedding of $\PSU(n)$.  This is the projectivisation of the Lagrangian $\U(n)$ from \cref{sscLagrangianFlagVariety}.  Its fundamental group is $\Z/n$, so since $X$ is monotone with minimal Chern number $n^2$ we see that $L$ is monotone and has minimal Maslov number divisible by $2n$.  It is parallelisable (it's a Lie group) and therefore orientable and spin.

This family was originally discovered by Amarzaya and Ohnita \cite{AmOh} and later rediscovered by Chiang \cite[Section 4]{Ch}, and we denote it by $\LCEL$.  Iriyeh \cite[Theorem 8, Proposition 27]{Iriyeh} proved that $\LCEL$ is wide over a field of characteristic $2$ if $n$ is a power of $2$ and narrow otherwise, and when $n$ is $3$ or $5$ $\LCEL$ is non-narrow over $\Z$.  Later, Evans and Lekili \cite[Example 7.2.3]{EL2} showed that $\LCEL$ is wide over a field of characteristic $p$ when $n$ is a power of $p$, and narrow if $n$ is not divisible by $p$.  Their arguments work with any relative spin structure and flat line bundle.  Recently this family has also been studied by Torricelli \cite{BrunMast}.

We complete these partial computations by showing

\begin{thm}\label{thmLCEL}
If $n$ is not a prime power then $\LCEL$ is narrow for any choice of coefficient field, relative spin structure, and flat line bundle $\mathcal{L}$.
\end{thm}

We shall actually deduce this from the following more general result

\begin{thm}\label{corLagPSU}  If $L$ is a $2$-periodic Lagrangian embedding of $\PSU(n)$ in a closed monotone symplectic manifold $X$, with $n \geq 2$ and $N_L \geq 2n$, then $N_L=2n$ and either:
\begen
\item\label{PSUitm2} $\K$ has prime characteristic $p$ and $n$ is a power of $p$, in which case $L^\flat$ is wide for all line bundles $\mathcal{L}$ and all relative spin structures.
\item\label{PSUitm1} Otherwise $L^\flat$ is narrow over $\K$ for all such $\mathcal{L}$ and $s$.
\end{enumerate}
\end{thm}
\begin{rmk}
By the same arguments as for $\LCEL$, any such $L$ inherits monotonicity from $X$ and is orientable and spin.  If the $2$-periodicity depends on the choice of background class for $X$ (e.g.~if the periodicity comes from an invertible element in $QH^*(X, b; \K)$ that only exists for a specific choice of $b$) then the result only applies to relative spin structures with this background class.
\end{rmk}

This immediately proves \cref{thmLCEL}, since for all $n$ we have
\[
QH^*(\C\P^n, b; \K) = \K[H] / (H^{n+1} -(-1)^{\ip{b}{[\text{line}]}}),
\]
where $H$ is the hyperplane class which is a degree $2$ invertible.  Here $[\text{line}]$ is the homology class of the curve which contributes to the quantum product $H^n \qcup H$.

Before proving \cref{corLagPSU} we need

\begin{lem}[{\cite[Th\'eor\`eme 11.4]{Borel}, \cite[Corollary 4.2]{BauBrow}}]\label{HPSU}
If $\Char \K = p > 0$ and $p^r$ is the greatest power of $p$ dividing $n$ then we have an isomorphism of graded algebras
\[
H^*(\PSU(n); \K) \cong \Lambda (x_1, x_3, \dots, \widehat{x}_{2p^r-1}, \dots, x_{2n-1}) \otimes R[y]/(y^{p^r}),
\]
where the $\Lambda$ denotes the exterior algebra over $\K$ generated by elements $x_{2j-1}$ of degree $2j-1$; except in the case $p=2$ and $r=1$ where the relation $x_1^2 = 0$ should be replaced by $x_1^2 = y$.  If $\Char \K = 0$ then as graded algebras
\begin{flalign*}
\label{eqHPSUz}
&& H^*(\PSU(n); \K) &\cong \Lambda (x_3, x_5, \dots, x_{2n-1}).&&\qed
\end{flalign*}
\end{lem}

We can now give the proof.

\begin{proof}[Proof of \cref{corLagPSU}]
The equality $N_L = 2n$ follows directly from \cref{thmExteriorAlgebraNL} with $\K=\Q$.

Let $p = \Char \K$ (prime or zero), and suppose first that $n$ is a power of $p$.  By \cref{HPSU} $H^*(\PSU(n); \K)$ is generated as a $\K$-algebra by elements of degree at most $2n-3$ (when $n=p=2$ this requires the relation $x_1^2=y$), so by \cref{BCWideNarrow}  we see that \ref{PSUitm2} holds.  This is how Evans and Lekili proved wideness in \cite{EL2}.

Now suppose that $n$ is not a power of $p$.  This time $H^*(\PSU(n); \K)$ is generated as a $\K$-algebra by elements of degree $\leq 2n-1$, so for any given $\mathcal{L}$ and $s$ \cref{BCWideNarrow} tells us that $L^\flat$ is either narrow or wide.  We need to rule out the latter, so suppose for contradiction that $L^\flat$ is wide.

Assume that $p$ is prime and let $p^r$ be its greatest power dividing $n$.  We then have
\begin{equation}
\label{eqPSUPF}
P_F(S) = (1+S^2+\dots+S^{2p^r-2}) \prod_{j=1, \dots, \widehat{p^r}, \dots, n} (1+S^{2j-1}) \mod S^{N_L} - 1,
\end{equation}
and from \cref{lemPeriodic} we get that $p^r \cdot 2^{n-1}$ is divisible by $N_L/2=n$ and that $P_F(\zeta)=0$ for any primitive $2n$th root of unity $\zeta$.  Setting $S=\zeta$ in \eqref{eqPSUPF}, the first bracketed term cannot vanish, so we obtain $2j-1=n$ for some $j$ in $\{1, \dots, \widehat{p^r}, \dots, n\}$.  Thus $n$ is odd, and the condition $n | p^r \cdot 2^{n-1}$ forces $n$ to divide $p^r$.  Since $p^r$ is a proper factor of $n$, this is impossible.  We are left to deal with $p=0$, but in this case the same argument applies if $p^r$ is interpreted as $1$.
\end{proof}

\begin{rmk}
Iriyeh's proof \cite[pp.~260--261]{Iriyeh} that $\LCEL$ is narrow in characteristic $2$ when $n$ is even but not a power of $2$ is similar in spirit.  He uses $2$-periodicity in the $\Z/2n$-grading, plus Poincar\'e duality, to deduce that if $n$ is even then $HF^*(\LCEL, \LCEL)$ has the same rank in every degree.  If $\LCEL$ were non-narrow, and hence wide, then this would mean that $2n$ divides the sum of the $\Z/2$ Betti numbers of $\LCEL$.  This sum is a power of $2$, so $n$ itself must be a power of $2$.
\end{rmk}


\section{Trivial vector bundles}
\label{secTrivialVBs}


\subsection{Restricting Chern classes}

Recall from \cref{sscFloerMorse} that for a monotone Lagrangian brane $L^\flat$ the PSS map $H^{< N_L}(L; \K) \rightarrow HF^*(L^\flat, L^\flat; \Lambda)$ intertwines the closed--open map
\[
\CO : QH^*(X; \Lambda) \rightarrow HF^*(L^\flat, L^\flat; \Lambda),
\]
with the classical restriction map $H^*(X; \K) \rightarrow H^*(L; \K)$ on classes of degree $<N_L$.  In particular

\begin{lem}
If $\alpha \in H^{\leq N_L}(X; \K)$ restricts to $0$ in $H^*(L; \K)$ then $\CO(\alpha)=0$.\hfill$\qed$
\end{lem}

An obvious corollary of this is the following observation.

\begin{prop}
\label{COChern}
Suppose $L^\flat \subset X$ is a monotone Lagrangian brane and $E \rightarrow X$ is a complex vector bundle whose restriction to $L$ is trivial.  For all $j$ with $1 \leq j < N_L / 2$ we have $\CO(c_j(E)) = 0$, where $c_j(E)$ is the $j$th Chern class of $E$.\hfill$\qed$
\end{prop}

The point of making this statement separately is that many algebraic varieties $X$ carry natural vector bundles whose Chern classes generate large parts of the cohomology (e.g.~as in \cref{sscLagrangianFlagVariety}).  Moreover, these classes can be easily manipulated using exact sequences and the splitting principle.


\subsection{Worked example: projective Stiefel manifolds}

We now apply this technique to the following family of examples, introduced to the author by Frol Zapolsky (who proved the wideness part of \cref{thmProjStiefel}); these spaces appear in his work \cite{ZapolskyGrassmannians} constructing quasi-morphisms on contactomorphism groups.  Fix positive integers $n$ and $k$ with $n \geq 2, k$.  We view $\C\P^{kn-1}$ as the projectivisation of the space of $k \times n$ matrices and let $\mathrm{PSU}(k)$ act by left multiplication in the obvious way.  This action is Hamiltonian, with moment map
\[
\ip{\mu([w])}{A} = -\frac{i}{2} \frac{\Tr w^\dag Aw}{\Tr w^\dag w}
\]
for all $A$ in $\mathfrak{su}(k)$ and all $k\times n$ matrices $Z$, and the symplectic reduction at the zero level is the (complex) Grassmannian $\Gr(k, n)$.  The set $\mu^{-1}(0)$ embeds as a Lagrangian
\[
L \subset X = \Gr(k, n)^- \times \C\P^{kn-1} = \Gr(n-k, n) \times \C\P^{kn-1},
\]
where ${}^-$ denotes reversal of the sign of the symplectic form, and is diffeomorphic to the (complex) projective Stiefel manifold, i.e.~the quotient of the Stiefel manifold $V$ parametrising unitary $k$-frames in $\C^n$ by the obvious action of $\U(1)$.  The case $k=n$ gives the family $\LCEL$ from \cref{sscLCEL}, whilst $k=1$ gives the diagonal $\Delta \subset (\C\P^{n-1})^- \times \C\P^{n-1}$.  We shall show

\begin{thm}
\label{thmProjStiefel}
If $p$ denotes $\Char \K$ and $p^r$ its greatest power dividing $n$ (interpreted as $1$ if $p=0$) then either: $k \leq p^r$, in which case $L$ is wide for all choices of relative spin structure and flat line bundle; or $k > p^r$, in which case $L$ is narrow for all such choices.
\end{thm}

Note that this is consistent with \cref{corLagPSU} when $k=n$, and with the fact that $\Delta$ is always wide when $k=1$ by \cref{BCWideNarrow} (for a specific choice of relative spin structure we also have the isomorphism $HF^*(\Delta, \Delta) \cong QH^*(\C\P^{n-1})$).  The $k=n$ case behaves slightly differently from the others (the Lagrangian is not simply connected, for example), so since we have already dealt with it by other means we henceforth exclude it.  The first task is to establish the basic properties of these Lagrangians:

\begin{lem}
$L$ is monotone and orientable, with $N_L=2n$.
\end{lem}

\begin{proof}
First consider the Stiefel manifold $V$.  Projecting a unitary frame to its first entry realises $V$ as a fibration over $S^{2n-1}$.  Projecting the fibre to its second entry then realises is as a fibration over $S^{2n-3}$, whose fibre is a fibration over $S^{2n-5}$, and so on, until we reach fibre $S^{2n-2k+1}$.  By iterating the long exact sequence in homotopy groups we see that $V$ is simply connected.  The long exact sequence in homotopy groups for the fibration $\U(1) \hookrightarrow V \twoheadrightarrow L$ then shows that $L$ is simply connected.  This proves that $L$ is orientable and (from the long exact sequence in homotopy for the pair $(X, L)$) that it is monotone if and only $X$ is monotone, with $N_L$ given by twice the minimal Chern number $N_X$.  To see that $N_X=n$, recall from \cref{FlagMinimalMaslov} that $\Gr(k, n)$ has minimal Chern number $n$, and we know that $\C\P^{kn-1}$ has minimal Chern number $kn$ (in fact, this is the special case $\Gr(1, kn)$), so their product has minimal Chern number $\gcd(n, kn)=n$.

It remains to prove monotonicity, and since $L$ is simply connected it suffices to show that $\C\P^{kn-1}$ and $\Gr(k, n)$ are monotone with the same monotonicity constant.  In fact, we claim that $\C\P^{kn-1}\times \C\P^{(n-k)n-1}$ (equipped with the sum of appropriately scaled Fubini--Study forms) and its $\PSU(k)\times\PSU(n-k)$-reduction $\Gr(k,n)\times\Gr(n-k, n)$ are monotone with the same monotonicity constant.  For this note that the Hamiltonian $\U(n)$-action on the space $\C^{n^2}$ of $n\times n$ matrices, with moment map \eqref{UnMomentMap}, restricts to Hamiltonian actions of both $\U(1)\times \U(1)$ and $\U(k)\times\U(n-k)$---in each case the first factor acts on the first $k$ rows and the second factor acts on the remaining $n-k$ rows.  The reductions are $\C\P^{kn-1}\times \C\P^{(n-k)n-1}$ and $\Gr(k, n) \times \Gr(n-k, n)$, and the symplectic form on the latter comes from the symplectic reduction of the former by the residual action of $\PSU(k) \times \PSU(n-k)$.  The claim then follows from \cref{lemRedMonotone}, by considering the monotone Lagrangian $\U(n)$ in $\C^{n^2}$ from \cref{sscLagrangianFlagVariety}: this Lagrangian is monotone (\cref{lemUnMonotone}), so both reductions are monotone with the same monotonicity constant.
\end{proof}

\begin{lem}
\label{lemProjStiefelRelSpin}
$L$ is relatively spin.
\end{lem}
\begin{proof}
Let $\pr_1$ and $\pr_2$ be the projections from $X=\Gr(k, n)^- \times \C\P^{kn-1}$ onto its two factors.  We claim that in fact $L$ carries a natural relative spin structure with background class $\pr_2^*w_2(\C\P^{kn-1})$, i.e.~that $\pr_2^*T\C\P^{kn-1}|_L \oplus TL$ carries a natural spin structure.  Fix a compatible almost complex structure $J$, and apply the argument from the proof of \cref{lemRedMonotone} (with $X$ replaced by $\C\P^{kn-1}$ and $Z$ by $L$) to see that $\pr_2^*T\C\P^{kn-1}|_L = J(\mathfrak{k}\cdot L) \oplus TL$.  Thus $\pr_2^*T\C\P^{kn-1}|_L \oplus TL$ is the doubled bundle $TL \oplus TL$ plus the trivial summand $J(\mathfrak{k}\cdot L)$.  The former has a natural spin structure from the doubling construction of \cref{lemRedRelSpin}, whilst the latter has a natural spin structure from its trivialisation, completing the proof.
\end{proof}

\begin{rmk}
Similarly $L$ has a natural spin structure with background class $\pr_1^*w_2(\Gr(k, n))$; now $\pr_1^*T\Gr(k, n)|_L \oplus TL$ is a quotient of the doubled bundle $TL \oplus TL$ by the trivial bundle $\mathfrak{k} \cdot L$.  In general $L$ need not be (absolutely) spin: take $k=1$ with $n$ odd for example.
\end{rmk}

We are now ready for

\begin{proof}[Proof of \cref{thmProjStiefel}]
Recall that $p$ denotes $\Char \K$ and $p^r$ the greatest power of $p$ dividing $n$ (taken to be $1$ if $p=0$).  Suppose first that $k > p^r$.  We need to show that $L$ is narrow for all choices of relative spin structure and flat line bundle.

Let $\pr_1$ and $\pr_2$ be as in the proof of \cref{lemProjStiefelRelSpin}, let $E$ be the tautological bundle over $\Gr(k, n)$ (of rank $k$) and let $F$ be the quotient $\underline{\C}^n/E$.  Consider the bundles
\[
E(1) \coloneqq \pr_1^* E \otimes \pr_2^* \mathcal{O}_{\C\P^{kn-1}}(1) \quad \text{and} \quad F(1) \coloneqq \pr_1^* F \otimes \pr_2^* \mathcal{O}_{\C\P^{kn-1}}(1).
\]
The short exact sequence
\[
0 \rightarrow E(1) \rightarrow \pr_2^* \mathcal{O}_{\C\P^{kn-1}}(1)^{\oplus n} \rightarrow F(1) \rightarrow 0
\]
gives $c(E(1)) \smile c(F(1)) = (1+H)^n$ in classical cohomology, where $H$ is the pullback of the hyperplane class from $\C\P^{kn-1}$, and the same then holds in quantum cohomology in degrees $<2n$.  In particular, we have
\begin{equation}
\label{ChernBinomial}
\sum_{l=0}^j c_l(E(1)) \qcup c_{j-l}(F(1)) = \binom{n}{j} H^{\qcup j} \quad \text{in } QH^*
\end{equation}
for $j=0, \dots, n-1$.

We claim $E(1)|_L$ is trivial so by \cref{COChern} we have $\CO(c_j(E(1)))=0$ for $1 \leq j < n$ (using the fact that $N_L=2n$).  Applying this to $\CO$ of \eqref{ChernBinomial}, we obtain
\[
\CO(c_j(F(1)))=\binom{n}{j} \CO(H)^{\qcup j}
\]
for $j=1, \dots, n-1$, and since $F(1)$ has rank $n-k$ we conclude that both sides vanish for $j=n-k+1, \dots, n-1$.  Setting $j=n-p^r$ we get
\[
\binom{n}{n-p^r} \CO(H)^{\qcup (n-p^r)} = 0 \quad \text{in } HF^*.
\]
The left-hand side is invertible since $\binom{n}{n-p^r}=\binom{n}{p^r}$ is coprime to $p=\Char \K$ and $H$ is invertible in $QH^*$, so the only possibility is that $HF^*=0$, i.e.~that $L$ is narrow as claimed.

It remains (for the case $k>p^r$) to show that $E(1)|_L$ is trivial. To see that this is the case note that the fibre of $E(1)$ over a point $(V \subset \C^n, l \subset \C^{kn}) \in X$, where $V$ is a subspace of rank $k$ and $l$ is a line, comprises linear maps $l \rightarrow V$.  If $(V, l)$ lies in $L$ then there exists a $k \times n$ matrix $A$ with orthonormal rows such that $l$ is the span of $A$ and $V$ is the span of the rows of $A$.  We therefore have $k$ natural maps $l \rightarrow V$ given by projecting an element $\lambda \cdot A$ of $l$ to each of its $k$ rows.  These maps define $k$ sections of $E(1)|_L$ which provide a trivialising frame, completing the proof of narrowness for $k>p^r$.

Now assume $k \leq p^r$.  We claim that $H^*(L; \K)$ is generated as an algebra in degrees $< 2n-1$ so wideness follows from \cref{BCWideNarrow}.  This claim can be seen from the computation of the full cohomology algebra in \cite[Theorem 1.1, Theorem 1.2(i)]{ProjStiefel}, noting that the smallest integer $N > n-k$ such that $\binom{n}{N} \not\equiv 0 \mod p$ is $n$.
\end{proof}

\begin{rmk}
The narrowness result could have been proved by periodicity considerations as in \cref{sscLCEL}, and conversely the results there on $\LCEL$ (the $k=n$ case above) could have been proved using these Chern class arguments.  It is interesting to note that whilst the former technique requires full knowledge of the Betti numbers of $L$, the latter relies on much softer calculations but is more dependent on the geometry of the ambient manifold $X$.
\end{rmk}


\section{The symplectic Gysin sequence}
\label{secGysin}


\subsection{The exact triangle}

In this section we illustrate the \emph{symplectic Gysin sequence} by filling in a missing computation from \cite{SmDCS} and studying a related example.  There are two distinct approaches to this theory in the literature, using different methods but leading to similar results: the Lagrangian circle bundle construction and Floer--Gysin sequence of Biran \cite{BirNonInt}, Biran--Cieliebak \cite{BirCieSubcrit} and Biran--Khanevsky \cite{BirKhan}, and Perutz's symplectic Gysin sequence associated to a spherically fibred coisotropic submanifold \cite{PerGysin}.  We shall follow the latter because Perutz explicitly deals with coefficient rings of characteristic other than $2$.  Strictly Perutz works with a Novikov variable that can have arbitrary real exponents, see \cite[Notation 1.5]{PerGysin}, but the monotonicity hypotheses mean that this is not strictly necessary and we can restrict to integer exponents as we have been using.

The setup is as follows.  $M$ and $N$ are closed symplectic manifolds and $L$ is a Lagrangian submanifold of $X \coloneqq M^- \times N$ (recalling that ${}^-$ denotes reversal of the sign of the symplectic form) such that the projections $\pr_M$ and $\pr_N$ to $M$ and $N$ respectively have the following properties: $\pr_M$ embeds $L$ in $M$; and $\pr_N$ exhibits $L$ as an oriented $S^k$-bundle over $N$.  As usual we assume that $L$ is monotone with minimal Maslov number $N_L$ at least $2$.  Perutz shows

\begin{thm}[{\cite[Theorem 6.2, Addendum 1.6]{PerGysin}}]
\label{thmGysin}
If $N_L = k+1$ and $L$ is equipped with (the trivial flat line bundle and) a relative spin structure whose background class is pulled back from $b_M \in H^2(M; \Z/2)$ then there is an exact triangle of $QH^*(N; \Lambda)$-modules
\begin{equation*}
\begin{tikzcd}
QH^{*-(k+1)}(N; \Lambda) \arrow{rr}{\widehat{e}\qcup} & & QH^*(N; \Lambda) \arrow{dl}{}
\\ & HF^*(L^\flat, L^\flat; \Lambda) \arrow{ul}{[1]}. &
\end{tikzcd}
\end{equation*}
The horizontal arrow is quantum product with $\widehat{e} = e + \nu T$, where $e$ is the Euler class of the oriented sphere bundle $L \rightarrow N$ and $\nu$ is the signed count of index $k+1$ discs through a point $x$ of $L$ which send a second boundary marked point to a \emph{global angular chain}.  The latter is a chain on $L$ which intersects a generic fibre of $\pr_N|_L$ in a single point and whose boundary is the union of the fibres over a chain in the base representing the Poincar\'e dual of the Euler class.  The $QH^*(N; \Lambda)$-action on $HF^*(L^\flat, L^\flat; \Lambda)$ is by pulling back to $QH^*(X; \Lambda)$ and using the closed--open map.
\end{thm}

\begin{rmk}
\label{rmkGysin}
In \cite{PerGysin} Perutz denotes $L$ by $\widehat{V}$, and works with Hamiltonian Floer cohomology $HF^*(\mu)$ of a symplectomorphism $\mu$ of $N$.  We take $\mu = \id_N$ so that $HF^*(\mu)$ becomes $QH^*(N)$.  His argument all goes through for background classes of the form $\pr_M^*b_M + \pr_N^*b_N$, with $b_N \in H^2(N; \Z/2)$ as long as $QH^*(N; \Lambda)$, is deformed to $QH^*(N, b_N; \Lambda)$.
\end{rmk}


\subsection{Worked example I: $\mathrm{SO}(3)$}
\label{sscGysinI}

In this subsection we revisit a monotone Lagrangian studied in \cite{SmDCS}.  There we showed that it is narrow except possibly when $\Char \K = 3$ or $\Char \K = 5$, depending on the choice of relative spin structure, and that it is wide in the $\Char \K = 3$ case.  The $\Char \K = 5$ case was left unresolved, but we shall now prove wideness in both cases using the Gysin sequence.

Take $M=(\C\P^1 \times \C\P^1)^-$ and $N=\C\P^1$, with each $\C\P^1$ given the Fubini--Study form, so that $X=(\C\P^1)^3$.  Let $L$ be the zero set of the moment map for the standard $\mathrm{SO}(3)$-action by rotation on the three $\C\P^1=S^2$ factors.  This comprises ordered triples of points on $S^2$ which form the vertices of an equilateral triangle on a great circle, so is precisely the lift of the Chiang Lagrangian \cite[Section 2]{Ch} in $\C\P^3$ under the branched cover $(\C\P^1)^3 \rightarrow \Sym^3 \C\P^1 \cong \C\P^3$.  Such a triangle is determined by two of its vertices, so $\pr_M$ embeds $L$ in $M$, and the projection $\pr_N$ to the third vertex exhibits $L$ as an orientable circle bundle over $N$, so we are in the setup of the Gysin sequence.  Note that $L$ is monotone (since $X$ is monotone and $\pi_1(L)$ is torsion), has $N_L \geq 2$ (since it's orientable---it's a free $\mathrm{SO}(3)$-orbit), and carries a \emph{standard} spin structure defined by the trivialisation of its tangent bundle coming from the infinitesimal $\mathrm{SO}(3)$-action.  Equip $L$ with the trivial line bundle and an arbitrary relative spin structure to give a brane $L^\flat$.

This Lagrangian is the `$N=3$' case of the main family of examples in \cite{SmDCS} and the computations of \cite[Sections 5.2--5.3]{SmDCS} show the following.  Given a point $(x_1, x_2, x_3)$ in $L$ there is a holomorphic disc $u_1 : (D, \pd D) \rightarrow (X, L)$ defined by
\[
u_1(z) = (x_1, RM_zR^{-1}x_2, RM_zR^{-1}x_3)
\]
where $R$ is any rotation of $\C\P^1$ sending $\infty$ to $x_1$ and $M_z$ is the map $\C \rightarrow \C$ given by multiplication by $z$.  This meets the $\mathrm{SO}(3)$-invariant divisor
\[
Z_{23} = \{([z_1], [z_2], [z_3]) \in X : [z_2]=[z_3]\},
\]
which is Poincar\'e dual to the class $H_2+H_3$ ($H_j$ denotes the pullback of the hyperplane class from the $j$th $\C\P^1$ factor), and the count of this disc computes $\CO(H_2+H_3)=\pm T \cdot 1_L$.  Moreover, this sign is positive for the standard spin structure.  Similarly there are discs $u_2$ and $u_3$ where the roles of the three factors are interchanged, and these meet divisors $Z_{13}$ and $Z_{12}$ and compute $\CO(H_1+H_3)$ and $\CO(H_1+H_2)$.  Up to reparametrisation, these are the only three index $2$ holomorphic discs through $(x_1, x_2, x_3)$ (strictly they are the `axial' index $2$ discs, but by \cite[Corollary 3.10]{EL1} all holomorphic index $2$ discs are axial), and their classes $A_1$, $A_2$ and $A_3$ freely generate $H_2(X, L; \Z)$

Since relative spin structures form a torsor for $H^2(X, L; \Z/2)$, and we have a distinguished choice---namely the standard spin structure---we can label each relative spin structure by a class $\eps \in H^2(X, L; \Z/2)$.  Letting $\eps_j = (-1)^{\ip{\eps}{A_j}}$ the above results can then be written as
\begin{equation}
\label{COeqn1}
\CO(H_1+H_2+H_3-H_j) = \eps_j T \cdot 1_L.
\end{equation}

\begin{lem}[{\cite[Theorem 5.4.5]{SmDCS}}, `$N=3$']
\label{P13Narrow}
If $HF^*(L^\flat, L^\flat; \Lambda)$ is non-zero then either $\Char \K=3$ and the $\eps_j$ are all equal, or $\Char \K=5$ and the $\eps_j$ are not all equal.
\end{lem}
\begin{proof}
By taking linear combinations of the relations \eqref{COeqn1} we obtain
\begin{equation}
\label{COeqn2}
\CO(2H_3) = (\eps_1+\eps_2-\eps_3)T \cdot 1_L.
\end{equation}
In $QH^*(X, \eps; \Lambda)$ we have $H_3^2 = \pm T$, where the sign is determined by pairing the background class (which is the image of $\eps$ in $H^2(X; \Z/2)$) with the class of a line on the third $\C\P^1$ factor, and since this line intersects $Z_{13}$ and $Z_{23}$ once each, but not $Z_{12}$, the sign is exactly $\eps_1\eps_2$.  Squaring \eqref{COeqn2}, we thus have
\[
4\eps_1\eps_2 \cdot 1_L = (\eps_1+\eps_2-\eps_3)^2T^2 \cdot 1_L.
\]
If $HF^* \neq 0$ we must then have
\[
3 = 2(\eps_1\eps_2+\eps_2\eps_3+\eps_3\eps_1)
\]
in $\K$.  If the $\eps_j$ coincide then the right-hand side is $6$, so $\Char \K$ must be $3$; otherwise the right-hand side is $-2$ and so $\Char \K$ must be $5$.
\end{proof}

In \cite[Theorem 5.7.3]{SmDCS} we used the symmetric group action that permutes the three $\C\P^1$ factors to show that $L^\flat$ is wide when $\Char \K = 3$ and the $\eps_j$ are equal.  We can now prove the main result

\begin{thm}
\label{P13Wide}
$L^\flat$ is wide in the cases allowed by \cref{P13Narrow}, i.e.~when $\Char \K = 3$ and the $\eps_j$ are all equal or when $\Char \K = 5$ and the $\eps_j$ are not.
\end{thm}
\begin{proof}
We shall apply \cref{thmGysin}.  The map $(b_M, b_N) \mapsto \pr_M^*b_M+\pr_N^*b_N$ is an isomorphism $H^2(M; \Z/2) \oplus H^2(N; \Z/2) \rightarrow H^2(X; \Z/2)$, so \cref{rmkGysin} allows us to take any relative spin structure on $L$, and the computation in the proof of \cref{P13Narrow} shows that $(-1)^{\ip{b_N}{[\C\P^1]}}=\eps_1\eps_2$, so the exact triangle we obtain is
\begin{equation*}
\begin{tikzcd}
\Lambda[H_3, T^{\pm 1}]/(H_3^2-\eps_1\eps_2T^2) \arrow{rr}{\widehat{e}\qcup} & & \Lambda[H_3, T^{\pm 1}]/(H_3^2-\eps_1\eps_2T^2) \arrow{dl}{}
\\ & HF^*(L^\flat, L^\flat; \Lambda) \arrow{ul}{[1]}. &
\end{tikzcd}
\end{equation*}
The class $\widehat{e}$ is the sum of the Euler class $2H_3$ with $\nu T$, where $\nu$ counts holomorphic index $2$ discs through a generic point $x$ of $L$, each weighted by the intersection of its boundary with a global angular chain.

One can explicitly construct a global angular chain and compute the value of $\nu$, but in fact we can use \cref{P13Narrow} to save us the trouble.  With respect to the basis $1, H_3$ of $QH^*(\C\P^1, b_N; \Lambda)$ as a $\Lambda$-module, the map $\widehat{e}\qcup$ has matrix
\[
\begin{pmatrix} \nu T & 2\eps_1\eps_2 T^2 \\ 2 & \nu T \end{pmatrix},
\]
with determinant $(\nu^2-4\eps_1\eps_2)T^2$.  In particular, $HF^* \neq 0$ if and only if $\nu^2-4\eps_1\eps_2$ vanishes in $\K$.  For all integers $\nu$ and for all $\eps_j \in \{\pm 1\}$, the quantity $\nu^2-4\eps_1\eps_2$ is never $\pm 1$, so we see that there is always some characteristic in which $HF^*$ non-zero.  By \cref{P13Narrow} we then conclude that $L^\flat$ is non-narrow in characteristic $3$ when the $\eps_j$ are all equal and in characteristic $5$ when they are not.  In each case, $H^*(L; \K)$ has rank $2$ so $L^\flat$ is automatically wide (there is only one potentially non-zero differential in the Oh spectral sequence and non-narrowness means this differential is zero).
\end{proof}

\begin{rmk}
In the wide cases we know that $HF^*$ is the cone on multiplication by $2H_3+\nu T$ so $2H_3$ acts as $-\nu T$.  Hence $\CO(2H_3)=-\nu T \cdot 1_L$, so by \eqref{COeqn2} we get that $\nu = \eps_3-\eps_1-\eps_2$ in $\K$.  This agrees with the explicit calculation of $\nu$ over $\Z$, by counting discs meeting the global angular chain.
\end{rmk}


\subsection{Worked example II: $L(4,1)$}
\label{sscGysinII}

We now consider the following closely-related example.  Take $M=(\C\P^2)^-$ and $N=\C\P^1$, each equipped with an appropriate multiple of the Fubini--Study form so that the product $X=M^-\times N$ is monotone.  Take the Hamiltonian $\mathrm{SO}(3)$-action on $X = (\Sym^2 \C\P^1) \times \C\P^1$ which rotates the $\C\P^1$'s, and let $L$ be the zero set of the moment map.  This is an $\mathrm{SO}(3)$-orbit comprising triples $(x_1, x_2, x_3)$ of points on the sphere, with $x_1$ and $x_2$ unordered, which form the vertices of an isosceles triangle with apex at $x_3$ of a specific angle.  The stabiliser of such a configuration is the group of order $2$ generated by the rotation through angle $\pi$ about $x_3$, so the orbit is diffeomorphic to the lens space $L(4, 1)$.  As before, it is monotone, orientable (hence has $N_L \geq 2$), and carries a standard spin structure.  Equipping $L$ with the trivial flat line bundle and an arbitrary relative spin structure to give a brane $L^\flat$, our goal is to compute the characteristics in which $L^\flat$ is wide.

\begin{rmk}
We can explicitly calculate the moment map and apex angle following the conventions of \cite[Sections 3.1--3.2]{Sm} but with the symplectic form on $\C\P^n$ scaled by $n+1$; this ensures that a complex line has area $(n+1)\pi$ and Chern number $n+1$, which gives monotonicity.  In detail, we take $x$, $y$ as the basis for the standard representation of $\SU(2)$ and consider the bases
\[
x^n, x^{n-1}y, x^{n-2}y^2, \dots, y^n \quad \text{and} \quad x^n, \sqrt{\binom{n}{1}}x^{n-1}y, \sqrt{\binom{n}{2}}x^{n-2}y^2, \dots, y^n
\]
for its $n$th symmetric power.  We call the corresponding homogeneous coordinates on $\C\P^n$ \emph{standard} and \emph{unitary} coordinates respectively.  If $z$ is the vector of unitary coordinates on $\C\P^n$ then the moment map $\mu$ for the $\SU(2)$-action is
\[
\ip{\mu([z])}{\xi} = \frac{(n+1)i}{2} \frac{z^\dag \phi(\xi) z}{z^\dag z}
\]
for all $\xi$ in $\mathfrak{su}(2)$, where $\phi(\xi)$ is the matrix for the the infinitesimal action in unitary coordinates.

In our case we take $w=(w_0, w_1, w_2)$ and $z=(z_0, z_1)$ as unitary coordinates on $\C\P^2$ and $\C\P^1$ respectively, and see that the moment map satisfies
\[
\ip{\mu([w], [z])}{\lb\begin{smallmatrix}i & 0 \\ 0 & -i\end{smallmatrix}\rb} = 3 \frac{|w_2|^2-|w_0|^2}{\lVert w \rVert^2} + \frac{|z_1|^2-|z_0|^2}{\lVert z \rVert^2}.
\]
The isosceles triangle with apex $0$ and `base' vertices $\pm \lambda$ corresponds to $w=(-\lambda^2, 0, 1)$ and $z=(0, 1)$, and if this lies in $\mu^{-1}(0)$ then
\[
3\frac{1-\lambda^4}{1+\lambda^4} + 1 = 0.
\]
This yields $\lambda^4 = 1/2$, and hence the apex angle is $\arccos (3-2\sqrt{2}) \approx 80^\circ$ using \cite[Equation (9)]{Sm}.
\end{rmk}

Again there are three holomorphic index $2$ discs through each point of $L$.  The analogues of $u_1$ and $u_2$ meet the $\mathrm{SO}(3)$-invariant divisor comprising triples of points on $\C\P^1$, the first two unordered, such that (at least) one of the first two points coincides with the third; this is given by
\[
\{([ax^2+bxy+cy^2], [dx+ey]) : ae^2+be(-d)+c(-d)^2=0\}
\]
so its Poincar\'e dual is $H_1+2H_3$, where $H_1$ and $H_3$ are the hyperplane classes on $\C\P^2$ and $\C\P^1$ respectively.  The analogue of $u_3$ meets the invariant divisor comprising triples of points where the unordered pair coincide, given by
\[
\{([ax^2+bxy+cy^2], [dx+ey]) : b^2-4ac=0\},
\]
Poincar\'e dual to $2H_1$.  Let $A_1$, $A_2$ and $A_3$ denote the homology classes of these discs.

\begin{lem}
$A_1$ and $A_3$ form a basis for $H_2(X, L; \Z)$.
\end{lem}
\begin{proof}
The long exact sequence of the pair gives a short exact sequence
\[
0 \rightarrow H_2(X; \Z) \cong \Z^2 \rightarrow H_2(X, L; \Z) \rightarrow H_1(L; \Z) \cong \Z/4 \rightarrow 0,
\]
whilst intersecting with the two divisors above gives a map $\theta : H_2(X, L; \Z) \rightarrow \Z^2$.  The latter sends $A_1$ and $A_3$ to $(1, 0)$ and $(0, 1)$ respectively, so it suffices to show it's injective.  Since $H_2(X; \Z)$ has index $4$ as a subgroup of $H_2(X, L; \Z)$, we have that $\theta$ is injective if and only if $\theta|_{H_2(X; \Z)}$ is injective and $\theta(H_2(X; \Z))$ has index $4$ in $\theta(H_2(X, L; \Z))$, and this is what we shall prove.

To show these two properties, note that lines on the $\C\P^2$ and $\C\P^1$ factors form a basis for $H_2(X, \Z)$, and are sent by $\theta$ to $(1, 2)$ and $(2, 0)$ respectively.  Thus $\theta|_{H_2(X; \Z)}$ is injective and has cokernel $\Z/4$.  Since $\theta$ itself is surjective we see that $\theta(H_2(X; \Z))$ has index $4$ in $\theta(H_2(X, L; \Z))$, so we're done.
\end{proof}

As before, we introduce signs $\eps_1$ and $\eps_3$ to parametrise the relative spin structure, and now the three discs compute (by \cite[Theorem 3.5.3]{SmDCS}) that
\[
\CO(H_1+2H_3)=2\eps_1T\cdot 1_L \quad \text{and} \quad \CO(2H_1)=\eps_3T\cdot 1_L.
\]
The relations in quantum cohomology, meanwhile, become $H_1^3=\eps_1 T^3$ and $H_3^2 = T^2$.

\begin{lem}
\label{P2P1Narrow}
$L^\flat$ is narrow unless: $\Char \K = 7$ and $\eps_1=\eps_3$; or $\Char \K = 3$ and $\eps_1 = -\eps_3$.
\end{lem}
\begin{proof}
Cubing the equality $\CO(2H_1)=\eps_3T\cdot 1_L$ gives $8\eps_1T \cdot 1_L = \eps_3T \cdot 1_L$, so if $HF^* \neq 0$ then $8\eps_1$ must be equal to $\eps_3$ in $\K$.
\end{proof}

Since we are in the setting of the Gysin sequence we can use it to prove

\begin{thm}
\label{P2P1Wide}
$L^\flat$ is wide when $\Char \K = 7$ and $\eps_1=\eps_3$ or when $\Char \K = 3$ and $\eps_1 = -\eps_3$.
\end{thm}
\begin{proof}
We argue as in \cref{P13Wide}.  The exact triangle is now
\begin{equation*}
\begin{tikzcd}
\Lambda[H_3, T^{\pm 1}]/(H_3^2-T^2) \arrow{rr}{\widehat{e}\qcup} & & \Lambda[H_3, T^{\pm 1}]/(H_3^2-T^2) \arrow{dl}{}
\\ & HF^*(L^\flat, L^\flat; \Lambda) \arrow{ul}{[1]} &
\end{tikzcd}
\end{equation*}
with $\widehat{e} = 4H_3+\nu T$, and the determinant of the map $\widehat{e}\qcup$ is $(\nu^2-16)T^2$.  Since $\nu^2-16$ is never $\pm 1$, as before there is always some characteristic in which $HF^*$ non-zero.  By \cref{P2P1Narrow} we deduce that $L^\flat$ is non-narrow in characteristic $7$ when $\eps_1=\eps_3$ and in characteristic $3$ when $\eps_1=-\eps_3$.  Again, in each case $H^*(L; \K)$ has rank $2$ so $L^\flat$ is automatically wide.
\end{proof}


\section{Quilt theory and the Chekanov tori}
\label{secQuilts}


\subsection{Lagrangian correspondences}
\label{sscQuiltSummary}

We now turn to the \emph{quilt theory} of Wehrheim and Woodward, set out in \cite{QuiltedHF} and subsequent papers by the same authors and by Ma'u-Wehrheim--Woodward (this theory is actually also the basis of Perutz's Gysin sequence).  We shall use their composition theorem to relate the Lagrangians studied in \cref{sscGysinI,sscGysinII} to the Chekanov tori in $\C\P^1\times\C\P^1$ and $\C\P^2$ respectively.

Recall that given symplectic manifolds $(X_j, \omega_j)$, a Lagrangian correspondence from $X_{j-1}$ to $X_j$ is a Lagrangian submanifold $L_{(j-1)j}$ of $X_{j-1}^- \times X_j$, where as above $X_{j-1}^-$ is shorthand for $(X_{j-1}, -\omega_{j-1})$.  These generalise both ordinary Lagrangians in $X_j$, when $X_{j-1}$ is a point, and symplectomorphisms from $X_{j-1}$ to $X_j=X_{j-1}$, when $L_{(j-1)j}$ is the graph.  The \emph{composition} of correspondences $L_{(j-1)j}$ and $L_{j(j+1)}$, written $L_{(j-1)j} \circ L_{j(j+1)}$, is the subset
\begin{equation}
\label{eqCorrComp}
\pi_{(j-1)(j+1)} \big( ( L_{(j-1)j} \times L_{j(j+1)} ) \cap ( X_{j-1}^- \times \Delta_{X_j} \times X_{j+1} ) \big) \subset X_{j-1}^- \times X_{j+1},
\end{equation}
where $\pi_{(j-1)(j+1)}$ is the projection
\[
X_{j-1}^- \times X_j \times X_j^- \times X_{j+1} \rightarrow X_{j-1}^- \times X_{j+1}
\]
and $\Delta_{X_j}$ is the diagonal in $X_j^- \times X_j$.  The correspondence is said to be embedded if the intersection in \eqref{eqCorrComp} is transverse and the restriction of $\pi_{(j-1)(j+1)}$ to this intersection is an embedding, in which case it is a Lagrangian correspondence from $X_{j-1}$ to $X_{j+1}$.

Under appropriate hypotheses, Wehrheim--Woodward define a `quilted' Floer cohomology for cycles of Lagrangian correspondences $X_0$ to $X_1$ to $\dots$ to $X_{r+1} = X_0$, and prove that it is invariant under replacing consecutive correspondences by their composition when it is embedded.  Moreover, when $r=1$ and $X_0$ is a point, so the cycle of correspondences is just a pair of Lagrangians in $X_1$, their theory reproduces the ordinary Lagrangian intersection Floer cohomology of the two Lagrangians.  For us the important result is:

\begin{thm}[{\cite[Theorem 6.3.1]{QuiltedHF}}]\label{CorrNonZero}
Suppose we have a Lagrangian correspondence $L_{01}$ from $X_0$ to $X_1$ and a Lagrangian $L_1$ in $X_1$ such that the composition $L_0 \coloneqq L_{01} \circ L_1$ is embedded.  Assume moreover that all of these manifolds are closed, oriented and monotone, with the same monotonicity constant, and that $\pi_1(X_0 \times X_1)$ is torsion.  If $HF^*(L_0, L_0) \neq 0$ then $HF^*(L_{01}, L_{01}) \neq 0$.
\end{thm}

We have been deliberately vague about the coefficients here: as stated the result only applies in characteristic $2$, and to move outside this setting we need the orientations constructed in \cite{QuiltOr}.  First we fix relative spin structures on $L_{01}$ and $L_1$ with background classes $b_0+w_2(X_0)+b_1$ and $b_1+w_2(X_1)$ for some $b_j \in H^2(X_j; \Z/2)$.  This induces a relative spin structure on $L_0$ with background class $b_0$ and it is with respect to these relative spin structures that \cref{CorrNonZero} holds.  Moreover, for these relative spin structures we have \cite[Equation (24)]{QuiltedHF}
\begin{equation}
\label{eqInd2Sum}
w(L_0)+w(L_{01})+w(L_1)=0,
\end{equation}
where $w$ denotes the signed count of index $2$ discs through a generic point of $L$.

\begin{rmk}
The proof of \cref{CorrNonZero} shows that $HF^*(L_0, L_0)$ is isomorphic to $HF^*(L_{01}^\mathrm{sh}, L_0\times L_1)$ in $X_0^- \times X_1$, where $L_{01}^\mathrm{sh}$ denotes $L_{01}$ with background class shifted by $w_2(X_0 \times X_1)$ in the sense of \cite[Remark 5.1.8]{QuiltOr}.  This shift reverses the sign of the count of index $2$ discs, so the differential on $CF^*(L_{01}^\mathrm{sh}, L_0\times L_1)$ squares to
\[
w(L_0 \times L_1) - w(L_{01}^\mathrm{sh}) = w(L_0) + w(L_1) + w(L_{01}) = 0.\qedhere
\]
\end{rmk}


\subsection{Worked example I: $\C\P^1 \times \C\P^1$}
\label{sscQuiltI}

Consider the Lagrangian $\mathrm{SO}(3)$-oribt $L \subset (\C\P^1)^3$ from \cref{sscGysinI}.  Our aim is to reprove \cref{P13Wide} using \cref{CorrNonZero}.  To do this we view $L$ as a Lagrangian correspondence $L_{01}$ from $X_0 = (\C\P^1 \times \C\P^1)^-$ to $X_1 = \C\P^1$, and consider its composition $L_0$ with the Clifford torus (equatorial circle) $L_1$ in $X_1$.  This is equivalent to performing symplectic reduction at the equatorial level set for the $S^1$-action on the third $\C\P^1$ factor by rotation about the vertical axis.  This composition is embedded, and $L_0$ is precisely the monotone Chekanov torus $\TCh$ in $X_0$ as presented by Entov--Polterovich \cite[Example 1.22]{EnPoRigid}.  It consists of ordered triples of points on the sphere which form the vertices of an equilateral triangle on a great circle, such that the third point is constrained to the equator.

\begin{rmk}
This torus was first discovered by Chekanov in $\R^4$ \cite{ChekTor}, and appears in both $\C\P^2$ and $\C\P^1 \times \C\P^1$ in many Hamiltonian isotopic guises.  Comparisons between various different constructions are given by Gadbled \cite{Gad} and Oakley--Usher \cite{OakUsh}.
\end{rmk}

The hypotheses of \cref{CorrNonZero} are satisfied so after choosing appropriate relative spin structures the non-vanishing of $HF^*(\TCh, \TCh)$ implies the non-vanishing of $HF^*(L, L)$.  By \cite[Proposition 6.1.4]{BCQS}, the former is equivalent to the vanishing of the homology class swept by the boundaries of the index $2$ discs through a generic point of $\TCh$, and these discs were explicitly computed (for a specific regular complex structure) by Chekanov--Schlenk \cite[Lemma 5.2]{ChSchTwist}.  There are exactly five such discs, in classes $D_1$, $S_1 - D_1 - D_2$, $S_1 - D_1$, $S_2 - D_1$ and $S_2 - D_1 + D_2$ in $H_2(\C\P^1 \times \C\P^1, \TCh)$, where $S_1$ and $S_2$ are the classes of the spheres in each factor and $D_1$ and $D_2$ are discs whose boundaries form a basis for $H_1(\TCh; \Z)$.  We shall show that the relative spin structures and signs work out to give \cref{P13Wide}.

First equip $L_{01}=L$ with the standard spin structure, so its three index $2$ discs all count positively, and equip $L_1$ with the trivial spin structure.  By the paragraph after \cref{CorrNonZero} this induces a relative spin structure on $L_0=\TCh$ with background class $0$, satisfying (by \eqref{eqInd2Sum})
\[
w(\TCh) = -w(L_{01})-w(L_1) = -5.
\]
This means the five discs computed by Chekanov--Schlenk must all count with negative signs, so the sum of their boundaries is $3\pd D_1$.  We deduce that in this case $L$ is non-narrow (hence wide) when $\Char \K = 3$.

Now suppose we change the relative spin structure on $L$ to the one with $\eps_1=\eps_2=-\eps_3=1$, with background class $H_1+H_2$, recalling that $H_j$ represents the hyperplane class on the $j$th factor of $X=(\C\P^1)^3$.  This reverses the sign of the $u_3$ disc, so $w(L)$ becomes $1+1-1=1$.  The induced relative spin structure on $\TCh$ then has background class $H_1+H_2$ and satisfies
\[
w(\TCh) = -w(L_{01})-w(L_1) = -3.
\]
Let $\delta \in H^2(\C\P^1 \times \C\P^1, \TCh; \Z/2)$ describe the difference between this relative spin structure and the one in the previous paragraph, with respect to which all discs counted negatively.  Since the background class is $H_1+H_2$, we can write $\delta$ as
\[
S_1^\vee + S_2^\vee + \delta_1 D_1^\vee + \delta_2 D_2^\vee
\]
for some $\delta_j \in \Z/2$, where $S_1^\vee, S_2^\vee, D_1^\vee, D_2^\vee$ is the basis of $H^2(\C\P^1\times\C\P^1, \TCh; \Z/2)$ dual to the basis $S_1, S_2, D_1, D_2$ of $H_2(\C\P^1\times\C\P^1, \TCh; \Z/2)$.  We then have
\[
-3 = w(\TCh) = -(-1)^{\delta_1}+(-1)^{\delta_1+\delta_2}+(-1)^{\delta_1}+(-1)^{\delta_1}+(-1)^{\delta_1+\delta_2} = (-1)^{\delta_1} (1+2(-1)^{\delta_2}),
\]
where the five terms correspond to the signs attached to the five discs in the order listed above, and we conclude that $(-1)^{\delta_1}=-1$ and $(-1)^{\delta_2}=1$.  The sum of the boundaries is thus $5\pd D_1$, so $L$ is non-narrow (hence wide) when $\Char \K = 5$.

To deal with the cases $\eps_1=\eps_2=\eps_3=-1$ and $\eps_1=\eps_2=-\eps_3=-1$ we can take the above arguments and shift all relative spin structures as in \cite[Remark 5.1.8]{QuiltOr}; this simply flips the signs of all index $2$ discs.  All other relative spin structures on $L$ can be obtained from the four already considered by permuting the $\C\P^1$ factors.


\subsection{Worked example II: $\C\P^2$}
\label{sscCP2}

We can do the same thing for the Lagrangian lens space $L$ in $\C\P^2 \times \C\P^1$ from \cref{sscGysinII}, composing with the equator on the $\C\P^1$ factor to give the Chekanov torus in $\C\P^2$ (again see the papers of Gadbled \cite{Gad} and Oakley--Usher \cite{OakUsh} for equivalences of various definitions).  Using the methods of Chekanov--Schlenk \cite{ChSchTwist}, Auroux \cite[Proposition 5.8]{Au} computed that again this torus bounds five index $2$ holomorphic discs through a generic point, this time in classes $D_1$, $S-2D_1-D_2$, $S-2D_1$, $S-2D_1$, $S-2D_1+D_2$, where $S$ is the class of a line on $\C\P^2$ and $D_1$ and $D_2$ are classes of discs whose boundaries form a basis of $H_1(\TCh; \Z)$.

When $L$ is equipped with the standard spin structure and the equator with the trivial spin structure, we obtain $w(\TCh)=-5$ so all discs count negatively and the sum of their boundaries is $7\pd D_1$.  Thus $L$ is wide when $\Char \K = 7$.

Now change the relative spin structure on $L$ to the one with $\eps_1=-\eps_3=1$.  This has zero background class and gives $w(L)=1$, so the induced relative spin structure on $\TCh$ has zero background class and $w(\TCh)=-3$.  The difference between this relative spin structure on $\TCh$ and the previous one is thus of the form $\delta_1D_1^\vee + \delta_2D_2^\vee$ with
\[
-3 = w(\TCh) = -(-1)^{\delta_1}-(-1)^{\delta_2}-1-1-(-1)^{\delta_2}.
\]
We deduce that $(-1)^{\delta_1}=-1$ and $(-1)^{\delta_2}=1$, so the sum of the boundaries is $9\pd D_1$ and $L$ is wide when $\Char \K = 3$.

\bibliography{homogbiblio}
\bibliographystyle{utcapsor2}

\end{document}